\documentclass [a4paper]{article}
\usepackage[cp1251]{inputenc}
\usepackage{amsmath,amssymb,amsthm}
\usepackage{cite}
 \begin{document}

\title{The quaternion   core inverse and its generalizations.}

\author {{\textbf{Ivan I. Kyrchei}} \thanks{Pidstrygach Institute for Applied Problems of Mechanics and Mathematics, NAS of Ukraine, Lviv,  Ukraine, E-mail address: kyrchei@online.ua}
}

\date{}
\maketitle

\begin{abstract}
In this paper we extend notions of the core inverse, core EP inverse, DMP inverse, and CMP inverse over the quaternion skew-field ${\mathbb{H}}$ and get their determinantal representations within the framework of the theory of column-row determinants previously introduced by the author. Since  the Moore-Penrose inverse and the Drazin  inverse are  necessary tools    to represent these generalized inverses,  we use   their   determinantal representations  previously obtained by  using row-column determinants.
As the special case, we give their determinantal representations for  matrices with complex entries as well.
A numerical example to illustrate the main result is given.
\end{abstract}

\textbf{AMS Classification}: 15A09; 15A15; 15B33

\textbf{Keywords:}
Core inverse; Core EP inverse; DMP inverse; CMP inverse;  Moore-Penrose inverse; Drazin inverse; quaternion matrix;
 noncommutative determinant.

 \section{Introduction}
\newtheorem{cor}{Corollary}[section]
\newtheorem{thm}{Theorem}[section]
\newtheorem{lem}{Lemma}[section]
\theoremstyle{definition}
 \newtheorem{defn}[thm]{Definition}
 \theoremstyle{remark}
 \newtheorem{rem}[thm]{Remark}
\newcommand{\rk}{\mathop{\rm rk}\nolimits}
\newcommand{\Ind}{\mathop{\rm Ind}\nolimits}

Let ${\rm
{\mathbb{R}}}$ and ${\rm
{\mathbb{C}}}$ be the real and complex number fields, respectively.
Throughout the paper, we denote  the set of all $m\times n$ matrices over the quaternion
skew field
$$ {\mathbb{H}}=\{a_{0}+a_{1}\mathbf{i}+a_{2}\mathbf{j}+a_{3}\mathbf{k}\,
|\,\mathbf{i}^{2}=\mathbf{j}^{2}=\mathbf{k}^{2}=\mathbf{ijk}=-1,~a_{1}, a_{2}, a_{3}\in
{\mathbb{R}}\},$$
by $ {\mathbb{H}}^{m\times n}$, and by $ {\mathbb{H}}^{m\times n}_{r}$ its subset of  matrices  with a rank $r$. For $ {\bf A}  \in
{\mathbb{H}}^{m\times n}$, the symbols  $ {\bf A}^{ *}$ and $\rk ({\bf A})$ will denote  the conjugate transpose   and the  rank
of  $ {\bf A}$, respectively.
A matrix $ {\bf A}  \in
{\mathbb{H}}^{n\times n}$ is Hermitian if ${\rm {\bf
A}}^{ *}  = {\bf A}$.

\emph{The Moore-Penrose inverse} of ${\bf A}\in{\rm {\mathbb{H}}}^{m\times n}$, denoted by ${\bf A}^{\dagger}$, is the unique matrix ${\bf X}\in{\rm {\mathbb{H}}}^{n\times m}$ satisfying the following equations,
 \begin{align}\label{eq1:MP}  {\rm {\bf A}}{\bf X}
{\rm {\bf A}} =& {\rm {\bf A}}; \\
                                  {\bf X} {\rm {\bf
A}}{\bf X}  = &{\bf X};\label{eq2:MP} \\
                                  \left( {\rm {\bf A}}{\bf X} \right)^{ *}  =& {\rm
{\bf A}}{\bf X}; \label{eq3:MP} \\
                                \left( {{\bf X} {\rm {\bf A}}} \right)^{ *}  =&{\bf X} {\rm {\bf A}}.\label{eq4:MP}  \end{align}
 For ${\bf A}\in{\rm {\mathbb{H}}}^{n\times n}$ with $k = \Ind{\bf A}$ the smallest positive number such that $\rk ({\bf A}^{k+1})=\rk ({\bf A}^{k})$ \emph{the
Drazin inverse} of ${\bf A}$, denoted by ${\bf A}^{d}$, is defined to be the unique matrix ${\bf X}$ that satisfying Eq. (\ref{eq2:MP})
and the following equations,
                                \begin{align}\label{eq:Dr1}
                                    {\bf A}{\bf X}  =&
{\bf X}{\bf A}; \\
{\bf X}{\bf A}^{k+1}= {\bf
         A}^{k}~~(6a);~~~~&{\text or}~~~~     {\bf A}^{k+1}
{\bf X}= {\bf
         A}^{k}~~(6b)\nonumber. \end{align}
 In particular, when $\Ind {\bf A}=1$,
then the matrix ${\bf X}$  is called
\emph{the group inverse} and is denoted by $ {\bf X}= {\bf
A}^{\#}$.
If $\Ind {\bf A}=0$, then ${\bf A}$ is
nonsingular, and $ {\bf A}^{d}\equiv  {\bf A}^{-1}$.

A matrix ${\bf A}$ satisfying the conditions $(i), (j),\ldots$ is called an $\{i, j, \ldots\}$-inverse of
${\bf A}$,  and is denoted by ${\bf A}^{( i,j,\ldots)}$. The set of matrices ${\bf A}^{( i,j,\ldots)}$ is denoted ${\bf A}\{ i,j,\ldots\}$. In particular, ${\bf A}^{(1)}$ is  an inner inverse,   ${\bf A}^{(2)}$ is  an outer inverse, and ${\bf A}^{(1,2)}$ is  an  reflexive inverse, ${\bf A}^{(1,2,3,4)}$ is the Moore-Penrose inverse, etc.

 ${\bf P}_A:= {\bf A}{\bf A}^{\dag}$ and ${\bf Q}_A:= {\bf A}^{\dag}{\bf A}$ are the orthogonal projectors onto the range of ${\bf A}$ and the range of ${\bf A}^*$, respectively.
For $ {\bf A} \in {\rm {\mathbb{C}}}^{n\times m}$, the symbols  $\mathcal{N} ({\bf A})$, and $\mathcal{R} ({\bf A})$ will denote  the kernel and the range space of ${\bf A}$, respectively.

The core inverse was introduced by Baksalary and Trenkler in \cite{baks}. Later, it was investigated by S. Malik in \cite{mal1} and S.Z. Xu et al. in \cite{xu},  among others.
\begin{defn}\cite{baks}\label{def:cor}
A matrix ${\bf X}\in  {\mathbb{C}}^{n\times n}$ is called the core inverse of ${\rm {\bf A}} \in {\mathbb{C}}^{n\times n}$
if it satisfies the conditions
$$
{\bf A}{\bf X}={\bf P}_A,~and~ \mathcal{R}({\bf X})=\mathcal{R}({\bf A}).$$
When such matrix ${\bf X}$ exists, it is denoted $ {\bf A}^{\tiny\textcircled{\#}}$.
\end{defn}
In 2014, the core inverse
was extended to the core-EP inverse defined by K. Manjunatha Prasad and  K.S. Mohana
\cite{pras}.
Other generalizations of the core inverse were recently introduced for $n\times n$ complex matrices,
namely BT inverses \cite{baks1},  DMP inverses \cite{mal1},  and CMP inverses \cite{meh}, etc.
The characterizations, computing methods, some applications of the core inverse and its generalizations were recently investigated in complex matrices and rings (see, e.g.,
\cite{chen,fer1,fer2,gao,gut,liu,miel,mos,pras1,rak,wang}).

The determinantal representation of the usual inverse is the matrix with cofactors
in entries that suggests a direct method of finding the inverse of a matrix. The same
is desirable for the generalized inverses. But, there are various expressions of determinantal representations of generalized inverses even for matrices with complex or
real entries,   (see, e.g.  \cite{bap,rao,sta1,sta2,kyr,kyr1,kyr_nov}).

Because of the non-commutativity of the quaternion algebra, difficulties arise already in determining of the quaternion determinant (see, e.g. \cite{as,coh}).
The problem of the determinantal representation of generalized
inverses   only now  can be solved thanks to the theory of row-column determinants introduced in \cite{kyr2,kyr3}.
Within the framework of the theory of column-row determinants, determinantal representations of various generalized quaternion inverses and generalized inverse solutions to quaternion matrix
equations have been derived by the author (see, e.g.\cite{kyr4,kyr5,kyr6,kyr7,kyr8,kyr9,kyr10,kyr11,kyr12,kyr13,kyr14,kyr15}) and by other researchers (see, e.g.\cite{song1,song2,song3,song5,song6}).

In this paper we distribute notions of the core inverse, core EP inverse, DMP inverse, and CMP inverse over the quaternion skew-field ${\mathbb{H}}$ and get their determinantal representations.  As the special cases, we give their determinantal representations for complex matrices as well.
We note that  a determinantal formula for the core EP generalized inverse in complex matrices has been derived in \cite{pras} based on the determinantal representation of an  reflexive inverse obtained in \cite{bap,rao}. In this paper we propose a new determinantal representation of the core EP  inverse  in complex matrices as well.

Due to noncommutativity of quaternions, the ring of quaternion matrices ${\rm M}\left( {n, {\mathbb{H}}}\right)$ has evidently some differences from the ring of complex matrices ${\rm M}\left( {n, {\mathbb{C}}}\right)$. So,  quaternion generalized core inverses could have some features that will be discussed below.

The paper is organized as follows.  In Section 2, we start with preliminary  introduction of  row-column determinants,  determinantal representations of the Moore-Penrose inverse and the Drazin inverse previously obtained within the framework of the theory   of  row-column determinants, and some provisions of quaternion vector spaces.  In Section 3, we give determinantal representations  of the core, core EP, DMP, and CMP inverses over the quaternion skew-field, namely  the right and left core inverses are established in Subsection 3.1, the core EP inverses in Subsection 3.2, the core DMP inverse and its dual in Subsection 3.3, and finally the CMP inverse in Subsection 3.4.
A numerical example to illustrate the main results is considered in   Section 4. Finally, in Section 5, the conclusions are drawn.

\section{Preliminaries. Elements of the theory of  row-column determinants.}
Suppose $S_{n}$ is the symmetric group on the set $I_{n}=\{1,\ldots,n\}$.
 Let  $ {\bf A} \in
{\mathbb{H}}^{n\times n}$.
Row determinants of $ {\bf A}$ along its each rows can be  defined as follows.
\begin{defn}\cite{kyr2}
\emph{The $i$th row determinant} of ${\rm {\bf A}}=(a_{ij}) \in
{\mathbb{H}}^{n\times n}$ is defined  for all $i =1,\ldots,n $
by putting
 \begin{align*}{\rm{rdet}}_{ i} {\rm {\bf A}} =&
\sum\limits_{\sigma \in S_{n}} \left( { - 1} \right)^{n - r}({a_{i{\kern
1pt} i_{k_{1}}} } {a_{i_{k_{1}}   i_{k_{1} + 1}}} \ldots   {a_{i_{k_{1}
+ l_{1}}
 i}})  \ldots  ({a_{i_{k_{r}}  i_{k_{r} + 1}}}
\ldots  {a_{i_{k_{r} + l_{r}}  i_{k_{r}} }}),\\
\sigma =& \left(
{i\,i_{k_{1}}  i_{k_{1} + 1} \ldots i_{k_{1} + l_{1}} } \right)\left(
{i_{k_{2}}  i_{k_{2} + 1} \ldots i_{k_{2} + l_{2}} } \right)\ldots \left(
{i_{k_{r}}  i_{k_{r} + 1} \ldots i_{k_{r} + l_{r}} } \right),\end{align*}
where $\sigma$ is the left-ordered  permutation. It means that its first cycle from the left  starts with $i$, other cycles  start from the left  with the minimal of all the integers which are contained in it,
$$i_{k_{t}}  <
i_{k_{t} + s}~~ \text{for all}~~ t = 2,\ldots,r,~~~s =1,\ldots,l_{t}, $$
and  the order of disjoint cycles (except for the first one)  is strictly conditioned by increase from left to right of their first elements, $i_{k_{2}} < i_{k_{3}}  < \cdots < i_{k_{r}}$.
\end{defn}
Similarly, for a column determinant along an arbitrary column, we have the following definition.
\begin{defn}\cite{kyr2}
\emph{The $j$th column determinant}
 of ${\rm {\bf
A}}=(a_{ij}) \in
{\mathbb{H}}^{n\times n}$ is defined for
all $j =1,\ldots,n $ by putting
 \begin{align*}{\rm{cdet}} _{{j}}  {\bf A} =&
\sum\limits_{\tau \in S_{n}} ( - 1)^{n - r}(a_{j_{k_{r}}
j_{k_{r} + l_{r}} } \ldots a_{j_{k_{r} + 1} j_{k_{r}} })  \ldots  (a_{j
j_{k_{1} + l_{1}} }  \ldots  a_{ j_{k_{1} + 1} j_{k_{1}} }a_{j_{k_{1}}
j}),\\
\tau =&
\left( {j_{k_{r} + l_{r}}  \ldots j_{k_{r} + 1} j_{k_{r}} } \right)\ldots
\left( {j_{k_{2} + l_{2}}  \ldots j_{k_{2} + 1} j_{k_{2}} } \right){\kern
1pt} \left( {j_{k_{1} + l_{1}}  \ldots j_{k_{1} + 1} j_{k_{1} } j}
\right), \end{align*}
\noindent where $\tau$ is the right-ordered  permutation. It means that its first cycle from the right  starts with $j$, other cycles  start from the right  with the minimal of all the integers which are contained in it,
$$j_{k_{t}}  < j_{k_{t} + s} ~~ \text{for all}~~ t = 2,\ldots,r,~~~s =1,\ldots,l_{t}, $$
and the order of disjoint cycles (except for the first one) is strictly conditioned by increase from right to left of their first elements,
 $j_{k_{2}}  < j_{k_{3}}  < \cdots <
j_{k_{r}} $.
\end{defn}
The row and column determinants have the following linear properties.
\begin{lem}\cite{kyr2}\label{lem:row_combin} If the $i$th row of
 $ {\bf A}\in {\mathbb{H}}^{n\times n}$ is a left linear combination
of  other row vectors, i.e. $ a_{i.} = \alpha_{1} {\bf b}_{1 } + \cdots + \alpha_{k}  {\bf b}_{k }$, where $
\alpha_{l} \in { {\mathbb{H}}}$ and ${\bf b}_{l }\in {\mathbb{H}}^{1\times n}$ for all $ l = {1,\ldots, k}$ and $ i = {1,\ldots, n}$, then
\[
 {\rm{rdet}}_{i}\,  {\bf A}_{i  .} \left(
\alpha_{1}  {\bf b}_{1 } + \cdots + \alpha_{k}
{\bf b}_{k }  \right)=\sum_l  \alpha_{l}{\rm{rdet}}_{i}\, {\bf A}_{i  .} \left(
  {\bf b}_{l } \right).
\]
\end{lem}
\begin{lem}\cite{kyr2}\label{lem:col_combin} If the $j$th column of
 $ A\in {\mathbb{H}}^{m\times n}$ is a right linear combination
of  other column vectors, i.e. $ {\bf a}_{.j} =  {\bf c}_{1}\alpha_{1} + \cdots +  {\bf c}_{k } \alpha_{k}$, where $
\alpha_{l} \in { {\mathbb{H}}}$ and ${\bf c}_{l }\in {\mathbb{H}}^{n\times1 }$ for all $ l = {1,\ldots, k}$ and $ j = {1,\ldots, n}$, then
\[
 {\rm{cdet}}_{j}\,  {\bf A}_{.j} \left(
  {\bf c}_{1 }\alpha_{1} + \cdots +
{\bf c}_{k }  \alpha_{k} \right)=\sum_l {\rm{cdet}}_{j}\,  {\bf A}_{.j} \left(
  {\bf c}_{l} \right) \alpha_{l}.
\]
\end{lem}
So,  an arbitrary $n\times n$ quaternion matrix inducts a set from $n$ row determinants and $n$ column determinants that are different in general. Only for
    Hermitian $ {\bf A}$, we have \cite{kyr2},
 $${\rm{rdet}} _{1}  {\bf A} = \cdots = {\rm{rdet}} _{n} {\bf
A} = {\rm{cdet}} _{1}  {\bf A} = \cdots = {\rm{cdet}} _{n}  {\bf
A} \in  {\mathbb{R}},$$ that enables to define \emph{the determinant
of a Hermitian matrix}  by putting
$\det {\bf A}: = {\rm{rdet}}_{{i}}\,
{\bf A} = {\rm{cdet}} _{{i}}\, {\bf A} $
 for all $i =1,\ldots,n$.

Properties of the
determinant of a Hermitian matrix are
similar to the properties of an usual (commutative) determinant and they have been completely explored by row-column determinants in \cite{kyr3}. In particular, for any matrix $ {\bf A} \in {\mathbb{H}}^{m\times n} $ it is proved \cite{kyr3} that $ \det({\bf A}^*{\bf A})=0$ iff some column of $ {\bf A}$ is a right linear combination of others, or some row of $ {\bf A}^*$ is a left linear combination of others. From this follows the definition of the \emph{determinantal rank} of a quaternion matrix $ {\bf A} $ as the largest possible size of a nonzero principal minor of its corresponding Hermitian matrix ${\bf A}^*{\bf A}$. It is shown that the row rank of a quaternion matrix $ {\bf A} \in {\mathbb{H}}^{m\times n} $ (that is a number of its left-linearly independent rows), the column rank (that is a number of its right-linearly independent columns) and its determinantal rank are equivalent herewith ${\rm rank}\, ({\bf A}^*{\bf A})={\rm rank}\, ({\bf A}{\bf A}^*)$.

For introducing determinantal representations of generalized inverses, the following notations will be used.

 Let $\alpha : = \left\{
{\alpha _{1} ,\ldots ,\alpha _{k}} \right\} \subseteq {\left\{
{1,\ldots ,m} \right\}}$ and $\beta : = \left\{ {\beta _{1}
,\ldots ,\beta _{k}} \right\} \subseteq {\left\{ {1,\ldots ,n}
\right\}}$ be subsets of the order $1 \le k \le \min {\left\{
{m,n} \right\}}$.  ${\bf A}_{\beta} ^{\alpha} $ denotes a submatrix of $ {\bf A}$ whose rows are indexed by
$\alpha$ and the columns indexed by $\beta$. So,  $ {\bf
A}_{\alpha} ^{\alpha}$ denotes a principal submatrix of $ {\bf A}$ with rows and columns indexed by $\alpha$.
 If $ {\bf A}$ is Hermitian, then
$\left| {\bf A}\right|_{\alpha} ^{\alpha}$ denotes the
corresponding principal minor of $\det {\rm {\bf A}}$.

Suppose $\textsl{L}_{ k,
n}: = {\left\{ {\,\alpha :\alpha = \left( {\alpha _{1} ,\ldots
,\alpha _{k}} \right),\,{\kern 1pt} 1 \le \alpha _{1} < \cdots
< \alpha _{k} \le n} \right\}}$ denotes
 the collection of strictly
increasing sequences of $1 \leq k\leq n$ integers chosen from $\left\{
{1,\ldots ,n} \right\}$. Then, for fixed $i \in \alpha $ and $j \in
\beta $, the collection of  sequences of row indexes that contain the index $i$ is denoted by $I_{r,\,m} {\left\{ {i} \right\}}: = \left\{
{\,\alpha :\alpha \in L_{r,m} ,i \in \alpha}  \right\}$, similarly, the collection of  sequences of column indexes that contain the index $j$ is denote by $ J_{r,\,n} {\left\{ {j} \right\}}: = {\left\{ {\,\beta :\beta
\in L_{r,n} ,j \in \beta}  \right\}}$.

 Let ${\rm {\bf a}}_{.j} $ be the $j$th column and ${\rm {\bf
a}}_{i.} $ be the $i$th row of ${\rm {\bf A}}$. Suppose ${\rm {\bf
A}}_{.j} \left( {{\rm {\bf b}}} \right)$ denotes the matrix obtained from
${\rm {\bf A}}$ by replacing its $j$th column with the column ${\rm {\bf
b}}$, and ${\rm {\bf A}}_{i.} \left( {{\rm {\bf b}}} \right)$ denotes the
matrix obtained from ${\rm {\bf A}}$ by replacing its $i$th row with the
row ${\rm {\bf b}}$.
 Denote by ${\rm {\bf a}}_{.j}^{*} $ and ${\rm {\bf
a}}_{i.}^{*} $ the $j$th column  and the $i$th row of  ${\rm
{\bf A}}^{*} $, respectively.
\begin{thm} \cite{kyr4}\label{th:det_rep_mp}
If $ {\bf A} \in  {\mathbb{H}}_{r}^{m\times n} $, then
the Moore-Penrose inverse  ${\rm {\bf A}}^{ \dag} = \left( {a_{ij}^{
\dag} } \right) \in  {\mathbb{H}}_{}^{n\times m} $ have the
following determinantal representations,
  \begin{equation}
\label{eq:cdet_repr_A*A}
 a_{ij}^{ \dag}  = {\frac{{{\sum\limits_{\beta
\in J_{r,n} {\left\{ {i} \right\}}} {{\rm{cdet}} _{i} \left(
{\left( { {\bf A}^{ *} {\bf A}} \right)_{. i}
\left( { {\bf a}_{\,.j}^{ *} }  \right)} \right)
 _{\beta} ^{\beta} } } }}{{{\sum\limits_{\beta \in
J_{r,n}} {{\left|  { {\bf A}^{ *}  {\bf A}}
\right| _{\beta} ^{\beta}  }}} }}},
\end{equation}
and
  \begin{equation}
\label{eq:rdet_repr_AA*} a_{ij}^{ \dag}  =
{\frac{{{\sum\limits_{\alpha \in I_{r,m} {\left\{ {j} \right\}}}
{{\rm{rdet}} _{j} \left( {( {\bf A} {\bf A}^{ *}
)_{j.} ( {\bf a}_{\,i.}^{ *} )} \right)_{\alpha}
^{\alpha} } }}}{{{\sum\limits_{\alpha \in I_{r,m}}  {{
{\left| { {\bf A} {\bf A}^{ *} } \right|
_{\alpha} ^{\alpha} } }}} }}}.
\end{equation}
\end{thm}
\begin{rem}\label{rem:unit_repr} For an arbitrary full-rank matrix $ {\bf A} \in  {\mathbb{H}}_{r}^{m\times n} $, a row-vector ${\bf b}\in  {\mathbb{H}}^{1\times m} $, and a column-vector ${\bf c}\in  {\mathbb{H}}^{n\times 1}$  we put, respectively,
\begin{itemize}
  \item if $r=m$, then
 \begin{align*} {\rm{rdet}} _{i} \left( {( {\bf A} {\bf A}^{ *}
)_{i.} \left( {\bf b}  \right)} \right)=&
{\sum\limits_{\alpha \in I_{m,m} {\left\{ {i} \right\}}}
{{\rm{rdet}} _{i} \left( {( {\bf A} {\bf A}^{ *}
)_{i.} \left( {\bf b}  \right)} \right)_{\alpha}
^{\alpha} } },\\
\det\left( { {\bf A} {\bf A}^{ *} } \right)=&{\sum\limits_{\alpha \in I_{m,m}}  {{
{\left| { {\bf A} {\bf A}^{ *} } \right|
_{\alpha} ^{\alpha} } }}}
 \end{align*}
for all $i=1,\ldots,m$;
 \item
   if $ r=n$, then
  \begin{align*}{\rm{cdet}} _{j} \left(
{\left( { {\bf A}^{ *} {\bf A}} \right)_{. j}
\left( {\bf c}  \right)} \right)=&
\sum\limits_{\beta
\in J_{n,n} {\left\{ {j} \right\}}} {{\rm{cdet}} _{j} \left(
{\left( { {\bf A}^{ *}  {\bf A}} \right)_{. j}
\left( {\bf c}  \right)} \right)
 _{\beta} ^{\beta} },
 \\\det\left(  { {\bf A}^{ *}  {\bf A}}
\right)=&{\sum\limits_{\beta \in
J_{n,n}} {{\left|  { {\bf A}^{ *} {\bf A}}
\right|_{\beta} ^{\beta}  }}} \end{align*}
for all $j=1,\ldots,n$.
\end{itemize}
\end{rem}

\begin{cor}\label{cor:det_repr_her_mp}
If $ {\bf A} \in  {\mathbb{H}}_{r}^{m\times n} $ is Hermitian,  then the
then
the Moore-Penrose inverse  ${\rm {\bf A}}^{ \dag} = \left( {a_{ij}^{
\dag} } \right) \in  {\mathbb{H}}_{}^{n\times m} $ have the
following determinantal representations,
  \begin{equation}
\label{eq:det_cdet_mp_her}
 a_{ij}^{ \dag}  = {\frac{{{\sum\limits_{\beta
\in J_{r,n} {\left\{ {i} \right\}}} {{\rm{cdet}} _{i} \left(
{\left( { {\bf A}^{2} } \right)_{. i}
\left( { {\bf a}_{\,.j}}  \right)} \right)
 _{\beta} ^{\beta} } } }}{{{\sum\limits_{\beta \in
J_{r,n}} {{\left|  { {\bf A}^{ 2}  }
\right| _{\beta} ^{\beta}  }}} }}},
\end{equation}
and
  \begin{equation}
\label{eq:det_rdet_mp_her} a_{ij}^{ \dag}  =
{\frac{{{\sum\limits_{\alpha \in I_{r,m} {\left\{ {j} \right\}}}
{{\rm{rdet}} _{j} \left( {( {\bf A}^{ 2}
)_{j.} ( {\bf a}_{\,i.} )} \right)_{\alpha}
^{\alpha} } }}}{{{\sum\limits_{\alpha \in I_{r,m}}  {{
{\left| {{\bf A}^{2} } \right|
_{\alpha} ^{\alpha} } }}} }}}.
\end{equation}

\end{cor}

\begin{cor}\label{cor:det_repr_proj_Q}
If $ {\bf A} \in  {\mathbb{H}}_{r}^{m\times n} $,  then the
projection matrix $ {\bf A}^{ \dag}  {\bf A} = : {\bf
Q}_{A} = \left( {q_{ij}} \right)_{n\times n} $ has the
determinantal representation
  \begin{equation}\label{eq:det_repr_proj_Q}
q_{ij} = {\frac{{{\sum\limits_{\beta \in J_{r,n} {\left\{ {i}
\right\}}} {{\rm{cdet}} _{i} \left( {\left( {{\bf A}^{ *}
{\bf A}} \right)_{.i} \left({\bf \dot{a}}_{.j} \right)}
\right)  _{\beta} ^{\beta} } }}}{{{\sum\limits_{\beta
\in J_{r,n}}  {{ \left| {{\bf A}^{ *}  {\bf
A}} \right|_{\beta}^{\beta} }}}} }},
 \end{equation}
\noindent where $ {\bf \dot{a}}_{.j} $ is the $j$th column of
${ {\bf A}^{ *}  {\bf A}} \in
{\mathbb{H}}^{n\times n}$.
\end{cor}
\begin{cor}\label{cor:det_repr_proj_P}
If $ {\bf A} \in  {\mathbb{H}}_{r}^{m\times n}$,  then the
projection matrix $ {\bf A} {\bf A}^{ \dag} = : {\bf
P}_{A} = \left( {p_{ij}} \right)_{m\times m} $ has the
determinantal representation

  \begin{equation}
\label{eq:det_repr_proj_P}
p_{ij} = {\frac{{{\sum\limits_{\alpha \in I_{r,m} {\left\{ {j}
\right\}}} {{{\rm{rdet}} _{j} {\left( {({\bf A} {\bf A}^{ *}
)_{j .} ({\bf \ddot{a}}  _{i  .} )}
\right)  _{\alpha} ^{\alpha} } }}}
}}{{{\sum\limits_{\alpha \in I_{r,m}} {{ {\left| {
{\bf A} {\bf A}^{ *} } \right| _{\alpha
}^{\alpha} }  }}} }}},
 \end{equation}
\noindent where ${\bf \ddot{a}} _{i.} $ is the $i$th row of $
{\bf A}{\bf A}^{*}\in  {\mathbb{H}}^{m\times m}$.
\end{cor}
The following corollary gives determinantal representations of the Moore-Penrose inverse and of  both projectors in complex matrices.
\begin{cor} \cite{kyr}\label{cor:det_repr_MP_c}
Let $ {\bf A} \in  {\mathbb{C}}_{r}^{m\times n} $. Then the
following determinantal representations are obtained
\begin{enumerate}
  \item[(i)] for the Moore-Penrose inverse  ${\rm {\bf A}}^{ \dag} = \left( {a_{ij}^{
\dag} } \right)_{n\times m} $,
  \begin{equation*}
 a_{ij}^{ \dag}  = {\frac{{{\sum\limits_{\beta
\in J_{r,n} {\left\{ {i} \right\}}} { \left|
{\left( { {\bf A}^{ *} {\bf A}} \right)_{. i}
\left( { {\bf a}_{\,.j}^{ *} }  \right)} \right|
 _{\beta} ^{\beta} } } }}{{{\sum\limits_{\beta \in
J_{r,n}} {{\left|  { {\bf A}^{ *}  {\bf A}}
\right| _{\beta} ^{\beta}  }}} }}}=
{\frac{{{\sum\limits_{\alpha \in I_{r,m} {\left\{ {j} \right\}}}
{ \left| {( {\bf A} {\bf A}^{ *}
)_{j.} ( {\bf a}_{\,i.}^{ *} )} \right|_{\alpha}
^{\alpha} } }}}{{{\sum\limits_{\alpha \in I_{r,m}}  {{
{\left| { {\bf A} {\bf A}^{ *} } \right|
_{\alpha} ^{\alpha} } }}} }}};
\end{equation*}
  \item[(ii)] for the
projector $ {\bf
Q}_{A} = \left( {q_{ij}} \right)_{n\times n} $,
  \begin{equation*}
q_{ij} = {\frac{{{\sum\limits_{\beta \in J_{r,n} {\left\{ {i}
\right\}}} { \left| {\left( {{\bf A}^{ *}
{\bf A}} \right)_{.i} \left({\bf \dot{a}}_{.j} \right)}
\right|  _{\beta} ^{\beta} } }}}{{{\sum\limits_{\beta
\in J_{r,n}}  {{ \left| {{\bf A}^{ *}  {\bf
A}} \right|_{\beta}^{\beta} }}}} }},
 \end{equation*}
where $ {\bf \dot{a}}_{.j} $ is the $j$th column of
${ {\bf A}^{ *}  {\bf A}}$;
  \item[(iii)] for the
projector $  {\bf
P}_{A} = \left( {p_{ij}} \right)_{m\times m} $,
 \begin{equation*}
p_{ij} = {\frac{{{\sum\limits_{\alpha \in I_{r,m} {\left\{ {j}
\right\}}} {{ {\left| {({\bf A} {\bf A}^{ *}
)_{j .} ({\bf \ddot{a}}  _{i  .} )}
\right|  _{\alpha} ^{\alpha} } }}}
}}{{{\sum\limits_{\alpha \in I_{r,m}} {{ {\left| {
{\bf A} {\bf A}^{ *} } \right| _{\alpha
}^{\alpha} }  }}} }}},
 \end{equation*}
where ${\bf \ddot{a}} _{i.} $ is the $i$th row of $
{\bf A}{\bf A}^{*}$.
\end{enumerate}
\end{cor}

\begin{lem}\cite{kyr5}\label{lem:det_repr_draz} If $ {\bf A} \in {\mathbb{H}}^{n\times n}$  with
$ \Ind  {\bf A}=k$ and  then the Drazin inverse  $ {\bf A}^{ d}=\left(a^{d}_{ij}\right)\in {\mathbb{H}}^{n\times n} $ possess the
determinantal representations \begin{itemize}
                                \item[(i)] if $ {\bf A}$ is an arbitrary, then
\begin{equation}
\label{eq:cdet_draz} a_{ij} ^{d}=
 {\frac{{ \sum\limits_{t = 1}^{n} {a}_{it}^{(k)}   {\sum\limits_{\beta \in J_{r,\,n} {\left\{ {t}
\right\}}} {{\rm{cdet}} _{t} \left( {\left(\left({\bf A}^{ 2k+1} \right)^{*}{\bf A}^{ 2k+1} \right)_{.\, t} \left( \hat{{\rm {\bf a}}}_{.\,j}
\right)} \right) _{\beta} ^{\beta} } }
}}{{{\sum\limits_{\beta \in J_{r,n}} {{\left| {\left({\bf A}^{ 2k+1} \right)^{*}{\bf A}^{ 2k+1}
}  \right|_{\beta} ^{\beta}}}} }}}
\end{equation}
 and \begin{equation}
\label{eq:rdet_draz} a_{ij} ^{d}=
{\frac{\sum\limits_{s = 1}^{n}\left({{\sum\limits_{\alpha \in I_{r,\,n} {\left\{ {s}
\right\}}} {{\rm{rdet}} _{s} \left( {\left( { {\bf A}^{ 2k+1} \left({\bf A}^{ 2k+1} \right)^{*}
} \right)_{\,. s} (\check{{\rm {\bf a}}}_{i\,.})} \right) {\kern 1pt} _{\alpha} ^{\alpha} } }
}\right){a}_{sj}^{(k)}}{{{\sum\limits_{\alpha \in I_{r,\,n}} {{\left| { { {\bf A}^{ 2k+1} \left({\bf A}^{ 2k+1} \right)^{*}
}
}  \right|_{\alpha} ^{\alpha}}}} }}}
\end{equation}
where    $\rk {\bf A}^{k+1} =\rk {\bf A}^{k}=r$,  $\hat{{\rm {\bf a}}}_{.j}$ is the $j$th column of $
({\bf A}^{ 2k+1})^{*} {\bf A}^{k}=:\hat{{\rm {\bf A}}}=
(\hat{a}_{ij})\in {\mathbb{H}}^{n\times n}$, and $\check{{{\rm {\bf a}}}}_{i.}$ the $t$th row of $
{\bf A}^{k}({\bf A}^{ 2k+1})^{*} =:\check{{\rm {\bf A}}}=
(\check{a}_{ij})\in {\mathbb{H}}^{n\times n}$;

                                \item [(ii)] if $ {\bf A}$ is Hermitian, then
\begin{equation}
\label{eq:dr_rep_cdet} a_{ij}^{d}  = {\frac{{{\sum\limits_{\beta
\in J_{r,n} {\left\{ {i} \right\}}} {{\rm{cdet}} _{i} \left(
{\left( {{\rm {\bf A}}^{k+1}} \right)_{. i} \left( {{\rm {\bf
a}}_{.j}^{ (k)} }  \right)} \right) _{\beta}
^{\beta} } } }}{{{\sum\limits_{\beta \in J_{r,n}} {{\left|
{{\bf A}^{k+1}
}  \right| _{\beta} ^{\beta}}}} }}},
\end{equation}
or
\begin{equation}
\label{eq:dr_rep_rdet} a_{ij}^{d}  = {\frac{{{\sum\limits_{\alpha
\in I_{r,n} {\left\{ {j} \right\}}} {{\rm{rdet}} _{j} \left(
{({\bf A}^{ k+1} )_{j.} \left( {\bf a}_{i.}^{ (k)} \right)}
\right)_{\alpha} ^{\alpha} } }}}{{{\sum\limits_{\alpha \in
I_{r,n}}  {{\left| {\bf A}^{k+1}    \right|_{\alpha} ^{\alpha}}}} }}},
\end{equation}
where $\rk {\bf A}^{k+1} =\rk {\bf A}^{k}=r$, ${\bf {a}}^{(k)} _{.j} $ and ${\bf {a}}^{(k)} _{i.} $ are the $j$th column and the $i$th row of ${\bf A}^{k}$, respectively.
 \end{itemize}
\end{lem}

\begin{cor}If $ {\bf A} \in {\mathbb{H}}^{n\times n}$  with
$ \Ind {\bf A}=1$, then the group inverse  $ {\bf A}^{\#} $ possess the
determinantal representations \begin{itemize}
                                \item[(i)] if ${\bf A}$ is an arbitrary, then
\begin{equation}
\label{eq:cdet_gr} a_{ij} ^{\#}=
 {\frac{{ \sum\limits_{t = 1}^{n} {a}_{it}  {\sum\limits_{\beta \in J_{r,n} {\left\{ {t}
\right\}}} {{\rm{cdet}} _{t} \left( {\left(\left({\bf A}^{ 3} \right)^{*}{\bf A}^{ 3} \right)_{. \,t} \left( \hat{{\bf a}}_{.j}
\right)} \right) _{\beta} ^{\beta} } }
}}{{{\sum\limits_{\beta \in J_{r,n}} {{\left| {\left({\bf A}^{3} \right)^{*}{\bf A}^{ 3}
}  \right|_{\beta} ^{\beta}}}} }}}
\end{equation}
 and \begin{equation}
\label{eq:rdet_gr} a_{ij} ^{\#}=
{\frac{\sum\limits_{s = 1}^{n}\left({{\sum\limits_{\alpha \in I_{r,n} {\left\{ {s}
\right\}}} {{\rm{rdet}} _{s} \left( {\left( { {\bf A}^{ 3} \left({\bf A}^{ 3} \right)^{*}
} \right)_{\,. s} (\check{ {\bf a}}_{i.})} \right) _{\alpha} ^{\alpha} } }
}\right){a}_{sj}}{{{\sum\limits_{\alpha \in I_{r,\,n}} {{\left| { { {\bf A}^{ 3} \left({\bf A}^{3} \right)^{*}
}
}  \right|_{\alpha} ^{\alpha}}}} }}}
\end{equation}
where   $\rk {\bf A}^{2} =\rk {\bf A} =
 r$,   $\hat{ {\bf a}}_{.j}$ is the $j$th column of $
({\bf A}^{ 3})^{*} {\bf A}=:\hat{ {\bf A}}=
(\hat{a}_{ij})\in {\mathbb{H}}^{n\times n}$ and $\check{\bf a}_{i.}$ the $i$th row of $
{\bf A}({\bf A}^{ 3})^{*} =:\check{{\bf A}}=
(\check{a}_{ij})\in {\mathbb{H}}^{n\times n}$;

                                \item [(ii)] if $ {\bf A}$ is Hermitian, then
\begin{equation}
\label{eq:gr_rep_cdet} a_{ij}^{\#}  = {\frac{{{\sum\limits_{\beta
\in J_{r,n} {\left\{ {i} \right\}}} {{\rm{cdet}} _{i} \left(
{\left( {{\bf A}^{2}} \right)_{.i} \left( { {\bf
a}_{.j}}  \right)} \right) _{\beta}
^{\beta} } } }}{{{\sum\limits_{\beta \in J_{r,n}} {{\left|
{ {\bf A}^{2}
}  \right| _{\beta} ^{\beta}}}} }}},
\end{equation}
or
\begin{equation}
\label{eq:gr_rep_rdet} a_{ij}^{\#}  = {\frac{{{\sum\limits_{\alpha
\in I_{r,n} {\left\{ {j} \right\}}} {{\rm{rdet}} _{j} \left(
{( {\bf A}^{ 2} )_{j.} \left( {\bf a}_{i.}\right)}
\right)_{\alpha} ^{\alpha} } }}}{{{\sum\limits_{\alpha \in
I_{r,n}}  {{\left|  {\bf A}^{2}  \right|_{\alpha} ^{\alpha}}}} }}}.
\end{equation}
where   $\rk {\bf A}^{2} =\rk {\bf A} =
 r$.
  \end{itemize}
\end{cor}

The following corollary gives determinantal representations of the Drazin inverse  in complex matrices.
\begin{cor} \cite{kyr}\label{cor:det_repr_Dr_c}
Let $ {\bf A} \in  {\mathbb{C}}^{n\times n} $ with
$ \Ind  {\bf A}=k$ and $\rk {\bf A}^{k+1} =\rk {\bf A}^{k}
 =r$. Then the Drazin inverse $ {\bf A}^{d} = \left( {a_{ij}^{
d} } \right) \in  {\mathbb{C}}_{}^{n\times n} $ has the
following determinantal representations
  \begin{equation}
\label{eq:dr_rep_det_c} a_{ij}^{d}  = {\frac{{{\sum\limits_{\beta
\in J_{r,n} {\left\{ {i} \right\}}} { \left|
{\left( {{\bf A}^{k+1}} \right)_{. i} \left( {{\rm {\bf
a}}_{.j}^{ (k)} }  \right)} \right| _{\beta}
^{\beta} } } }}{{{\sum\limits_{\beta \in J_{r,n}} {{\left|
{{\bf A}^{k+1}
}  \right| _{\beta} ^{\beta}}}} }}}=
{\frac{{{\sum\limits_{\alpha
\in I_{r,n} {\left\{ {j} \right\}}} { \left|
{({\bf A}^{ k+1} )_{j.} \left( {\bf a}_{i.}^{ (k)} \right)}
\right|_{\alpha} ^{\alpha} } }}}{{{\sum\limits_{\alpha \in
I_{r,n}}  {{\left| {\bf A}^{k+1}    \right|_{\alpha} ^{\alpha}}}} }}},
\end{equation}
where ${\bf {a}}^{(k)} _{.j} $ and ${\bf {a}}^{(k)} _{i.} $ are the $j$th column and the $i$th row of ${\bf A}^{k}$, respectively.
  \end{cor}
It is clear that the determinantal representations of the group inverse  in complex matrices have been obtained from (\ref{eq:dr_rep_det_c}) by putting $k=1$.

Due to quaternion-scalar multiplying on the right, quaternion column-vectors  form a right  vector $\mathbb{H}$-space, and, by quaternion-scalar multiplying on the left, quaternion row-vectors  form a left  vector $\mathbb{H}$-space denoted by $\mathcal{H}_{r}$ and $\mathcal{H}_{l}$, respectively. It can be shown that $\mathcal{H}_{r}$ and $\mathcal{H}_{l}$ possess corresponding $\mathbb{H}$-valued inner products by putting $\langle \mathbf{x},\mathbf{y}\rangle_{r}=\overline{y}_1x_{1}+\cdots+\overline{y}_{n}x_{n}$ for $\mathbf{x}=\left(x_{i}\right)_{i=1}^{n}, \mathbf{y}=\left(y_{i}\right)_{i=1}^{n}\in \mathcal{H}_{r}$, and $\langle \mathbf{x},\mathbf{y}\rangle_{l}=x_{1}\overline{y}_1+\cdots+x_{n}\overline{y}_{n}$ for $\mathbf{x}, \mathbf{y}\in \mathcal{H}_{l}$ that satisfy the  inner product relations, namely, conjugate symmetry, linearity, and positive-definiteness but with specialties
 \begin{equation*}\begin{gathered}\langle \mathbf{x}\alpha+\mathbf{y}\beta,\mathbf{z} \rangle=\langle \mathbf{x},\mathbf{z}\rangle\alpha+\langle \mathbf{y},\mathbf{z}\rangle\beta \,\,\mbox{ when}\,\, \mathbf{x}, \mathbf{y}, \mathbf{z} \in \mathcal{H}_{r},\\
          \langle \alpha\mathbf{x}+\beta\mathbf{y},\mathbf{z} \rangle=\alpha\langle \mathbf{x},\mathbf{z}\rangle+\beta\langle \mathbf{y},\mathbf{z}\rangle \,\, \mbox{ when}\,\, \mathbf{x}, \mathbf{y}, \mathbf{z} \in \mathcal{H}_{l},\end{gathered}\end{equation*}
 for any $\alpha, \beta \in {\mathbb{H}}$.

So, an arbitrary quaternion matrix induct vector spaces that introduced  by the following definition.
\begin{defn}
For an arbitrary matrix over the quaternion skew field, ${\bf A}\in  {\mathbb{H}}^{m\times n}$, we denote by
\begin{itemize}
  \item  $\mathcal{R}_{r}({\rm {\bf A}})=\{ {\bf y}\in {\mathbb{H}}^{m\times 1} : {\bf y} = {\bf A}{\bf x},\,  {\bf x} \in {\mathbb{H}}^{n\times 1}\},$  the right column  space of ${\bf A}$,
  \item  $\mathcal{N}_{r}({\rm {\bf A}})=\{ {\bf x}\in {\mathbb{H}}^{n\times 1} : \,\, {\bf A}{\bf x}=0\}$,  the right null space of  ${\bf A}$,
  \item $\mathcal{R}_{l}({\rm {\bf A}})=\{ {\bf y}\in {\mathbb{H}}^{1\times n} : \,\,{\bf y} = {\bf x}{\bf A},\,\,  {\bf x} \in {\mathbb{H}}^{1\times m}\}$, the  left row space of ${\bf A}$,
  \item  $\mathcal{N}_{l}({\rm {\bf A}})=\{ {\bf x}\in {\mathbb{H}}^{1\times m} : \,\, {\bf x}{\bf A}=0\}$,  the left null space of  ${\bf A}$.
\end{itemize}
\end{defn}

\section{Determinantal representations of the core inverse and its generalizations}

\subsection{Determinantal representations of the core inverses}

Due to   quaternion noncommutativity, Definition \ref{def:cor} of the core inverse can be expand to matrices over ${\mathbb{H}}$ as follows.
\begin{defn}
A matrix ${\bf X} \in  {\mathbb{H}}^{n\times n}$ is called \emph{the right core inverse} of ${\bf A}  \in  {\mathbb{H}}^{n\times n}$
if it satisfies the conditions
$$
{\bf A}{\bf X}={\bf P}_A,~and~ \mathcal{R}_r({\bf X})=\mathcal{R}_r({\bf A}).$$
When such matrix ${\bf X}$ exists, it is denoted $ {\bf A}^{\tiny\textcircled{\#}}$.
\end{defn}
\begin{defn}\label{def:lcor}
A matrix ${\bf X} \in  {\mathbb{H}}^{n\times n}$ is called \emph{the left core inverse} of $ {\bf A} \in  {\mathbb{H}}^{n\times n}$
if it satisfies the conditions
$$
{\bf X}{\bf A}={\bf Q}_A,~and~ \mathcal{R}_l({\bf X})=\mathcal{R}_l({\bf A}).$$
When such matrix ${\bf X}$ exists, it is denoted $ {\bf A}_{\tiny\textcircled{\#}}$.
\end{defn}
\begin{rem}
A definition similar to Definition \ref{def:lcor} has been introduced for  $ {\bf A}\in{\mathbb{C}}^{n\times n}$ in \cite{zh_arx} given that
${\bf P}_{A^*}={\bf A}^*({\bf A}^*)^{\dag}=\left({\bf A}^{\dag}{\bf A}\right)^*={\bf A}^{\dag}{\bf A}={\bf Q}_A$, and $\mathcal{R}_l({\bf A})=\mathcal{R}_r({\bf A}^*)$. In \cite{zh_arx}, $ {\bf A}_{\tiny\textcircled{\#}}$ is called the dual core inverse of ${\bf A}$.
\end{rem}
Due to \cite{baks}, we introduce the following sets of quaternion matrices
\begin{align*}
{\mathbb H}_n^{\mathrm{CM}}=&\{{\bf A}\in  {\mathbb{H}}^{n\times n}~:~\rk {\bf A}^2=\rk {\bf A}\},\\
{\mathbb H}_n^{\mathrm{EP}}=&\{{\bf A}\in  {\mathbb{H}}^{n\times n}~:~ {\bf A}^{\dag}{\bf A}= {\bf A}{\bf A}^{\dag}\}=\{\mathcal{R}_r({\bf A})=\mathcal{R}_r({\bf A}^*)\}.
\end{align*}
The matrices from ${\mathbb H}_n^{\mathrm{CM}}$ are called group matrices or core matrices.
If ${\bf A}\in{\mathbb H}_n^{\mathrm{EP}},$  then clearly ${\bf A}^{\dag}={\bf A}^{\#}$.
Similarly as for complex matrices, the
core inverses of a square quaternion matrix ${\bf A}\in  {\mathbb{H}}^{n\times n}$ exist if and only if ${\bf A}\in {\mathbb H}_n^{\mathrm{CM}}$ or $\Ind{\bf A}=1$. Moreover, if ${\bf A}$ is
non-singular, $\Ind{\bf A}=0$, then its core inverses are the usual inverse.

Due to \cite{baks}, the following representations of right and left core inverses can be extended to quaternion matrices.
\begin{lem}\label{lem:rep_cor1}Let ${\bf A}\in {\mathbb H}_n^{\mathrm{CM}}$. Then
 \begin{align} {\bf A}^{\tiny\textcircled{\#}}=&{\bf A}^{\#}{\bf A}{\bf A}^{\dag},\label{eq:rep_rcor}\\
{\bf A}_{\tiny\textcircled{\#}}=&{\bf A}^{\dag}{\bf A}{\bf A}^{\#}\label{eq:rep_lcor}
\end{align}
\end{lem}
\begin{lem}\label{lem:rep_cor2}Let ${\bf A}\in {\mathbb H}_n^{\mathrm{CM}}$. Then
\begin{enumerate}
  \item[(i)] $ {\bf A}^{\tiny\textcircled{\#}}, {\bf A}_{\tiny\textcircled{\#}}\in{\mathbb H}_n^{\mathrm{EP}},
$

  \item[(ii)] $ \left({\bf A}^{\tiny\textcircled{\#}}\right)^{\dag}={\bf A}{\bf P}_A,~\left({\bf A}_{\tiny\textcircled{\#}}\right)^{\dag}={\bf Q}_A{\bf A},
$

  \item[(iii)] $\left({\bf A}^{\tiny\textcircled{\#}}\right)^{\#}={\bf A}{\bf P}_A,~\left({\bf A}_{\tiny\textcircled{\#}}\right)^{\#}={\bf Q}_A{\bf A},
$

  \item[(iv)]${\bf A}^{\tiny\textcircled{\#}},~{\bf A}_{\tiny\textcircled{\#}}\in{\bf A}\{1,2\}
$

  \item [(v)] $\left({\bf A}^{\tiny\textcircled{\#}}\right)^2{\bf A}={\bf A}^{\#},~{\bf A}\left({\bf A}_{\tiny\textcircled{\#}}\right)^2={\bf A}^{\#},
$

  \item[(vi)] $\left({\bf A}^{\tiny\textcircled{\#}}\right)^m=\left({\bf A}^m\right)^{\tiny\textcircled{\#}},~{\bf A}\left({\bf A}_{\tiny\textcircled{\#}}\right)^m={\bf A}\left({\bf A}^m\right)_{\tiny\textcircled{\#}},
$

  \item[(vii)] $\left({\bf A}^{\tiny\textcircled{\#}}\right){\bf A}={\bf A}^{\#}{\bf A},~{\bf A}\left({\bf A}_{\tiny\textcircled{\#}}\right)={\bf A}{\bf A}^{\#}.
$
\end{enumerate}
\end{lem}
\begin{rem} In Theorems \ref{th:detrep_rcor} and \ref{th:detrep_lcor}, we will suppose that ${\bf A}\in {\mathbb H}_n^{\mathrm{CM}}$ but ${\bf A}\notin {\mathbb H}_n^{\mathrm{EP}}$.
Since ${\bf A}\in {\mathbb H}_n^{\mathrm{CM}}$ and ${\bf A}\in {\mathbb H}_n^{\mathrm{EP}}$ (in particular, ${\bf A}$ is Hermitian), then it follows from Lemma \ref{lem:rep_cor1} and the definitions of the Moore-Penrose inverse and group inverse that ${\bf A}^{\tiny\textcircled{\#}}={\bf A}_{\tiny\textcircled{\#}}={\bf A}^{\#}={\bf A}^{\dag}$.
\end{rem}
\begin{thm}\label{th:detrep_rcor}Let ${\bf A}\in  {\mathbb H}_n^{\mathrm{CM}}$, $\rk{\bf A}^{2}=\rk{\bf A}=s$. Then its right core inverse  ${\bf A}^{\tiny\textcircled{\#}}=\left(a_{ij}^{\tiny\textcircled{\#},r}\right)$ has the following determinantal representation
 \begin{align}\label{eq:detrep_rcor}
&a_{ij}^{\tiny\textcircled{\#},r}=
 {\frac{ \sum\limits_{\alpha \in I_{s,n}{\left\{ j
\right\}}}{{\rm{rdet}}_{j} {\left( {({\bf A} {\bf A}^{ *}
)_{j .} ({\bf \tilde{u}}_{i .} )}
\right)_{\alpha} ^{\alpha} } }}
{{{\sum\limits_{\alpha \in I_{s,n}} {{\left| {\left( { {\bf A}^{ 3} \left({\bf A}^{3} \right)^{*}
} \right)_{\alpha} ^{\alpha}
}  \right|}}}
{\sum\limits_{\alpha \in I_{s,n}} {{ {\left| {
{\bf A} {\bf A}^{ *} } \right|_{\alpha
}^{\alpha} }  }}}
 }}},
\end{align}
 where
${\bf \widetilde{u}}_{i .} $ is the $i$th row of ${\widetilde{\bf U}}:={\bf U}{\bf A}^{ 2}{\bf A}^{*}$, and ${\bf U}=(u_{if})\in {\mathbb H}^{n\times n}$ such that
   \begin{align*}u_{if}={{\sum\limits_{\alpha \in I_{s,n} {\left\{ {f}
\right\}}} {{\rm{rdet}} _{f} \left( {\left( { {\bf A}^{ 3} \left({\bf A}^{ 3} \right)^{*}
} \right)_{. f} ({\check {\bf a}}_{i\,.})} \right) _{\alpha} ^{\alpha} } }
},
\end{align*}
 where
 ${\check {\bf a}}_{i\,.}$ is the $i$th row of $\check{{\rm {\bf A}}}=
{\bf A}({\bf A}^{ 3})^{*}$.
\end{thm}
\begin{proof}
By (\ref{eq:rep_rcor}),
\begin{equation*}
a_{ij}^{\tiny\textcircled{\#},r} =  \sum\limits_{l = 1}^{n} {a}_{il}^{\#} p_{lj}.
 \end{equation*}
Using (\ref{eq:cdet_gr}) for the determinantal representation of ${\bf A}^{\#}$ and (\ref{eq:det_repr_proj_P}) for the determinantal representation of ${\bf P}_A={\bf A}{\bf A}^{\dag}$, we obtain
\begin{align*}
&a_{ij}^{\tiny\textcircled{\#},r}= \sum\limits_{l = 1}^{n}{\frac{\sum\limits_{f = 1}^{n}\left({{\sum\limits_{\alpha \in I_{s,n} {\left\{ {f}
\right\}}} {{\rm{rdet}} _{f} \left( {\left( { {\bf A}^{ 3} \left({\bf A}^{ 3} \right)^{*}
} \right)_{. f} (\check{ {\bf a}}_{i\,.})} \right) _{\alpha} ^{\alpha} } }
}\right){a}_{fl}}{{{\sum\limits_{\alpha \in I_{s,n}} {{\left|  { {\bf A}^{ 3} \left({\bf A}^{3} \right)^{*}
}  \right| _{\alpha} ^{\alpha}
}}} }}}\times\\
&~~~~~~~~~~~~~~~~~{\frac{{{\sum\limits_{\alpha \in I_{s,n} {\left\{ {j}
\right\}}} {{{\rm{rdet}} _{j} {\left( {({\bf A} {\bf A}^{ *}
)_{j .} ({\bf \ddot{a}}  _{l  .} )}
\right)  _{\alpha} ^{\alpha} } }}}
}}{{{\sum\limits_{\alpha \in I_{s,n}} {{ {\left| {
{\bf A} {\bf A}^{ *} } \right| _{\alpha
}^{\alpha} }  }}} }}}=\\
&{\frac{\sum\limits_{f = 1}^{n}{{\sum\limits_{\alpha \in I_{s,n} {\left\{ {f}
\right\}}} {{\rm{rdet}} _{f} \left( {\left( { {\bf A}^{ 3} \left({\bf A}^{ 3} \right)^{*}
} \right)_{. f} ({\check {\bf a}}_{i\,.})} \right) _{\alpha} ^{\alpha} } }
}
{{\sum\limits_{\alpha \in I_{s,n} {\left\{ {j}
\right\}}} {{{\rm{rdet}} _{j} {\left( {({\bf A} {\bf A}^{ *}
)_{j .} ({\bf \tilde{a}}  _{f  .} )}
\right)  _{\alpha} ^{\alpha} } }}}
}}{{{\sum\limits_{\alpha \in I_{s,n}} {{\left| {\left( { {\bf A}^{ 3} \left({\bf A}^{3} \right)^{*}
} \right) _{\alpha} ^{\alpha}
}  \right|}}}
{\sum\limits_{\alpha \in I_{s,n}} {{ {\left| {
{\bf A} {\bf A}^{ *} } \right| _{\alpha
}^{\alpha} }  }}}
 }}},
\end{align*}
where ${\check {\bf a}}_{i\,.}$ is the $i$th row of $\check{{\bf A}}=
{\bf A}({\bf A}^{ 3})^{*}$ and ${\bf \tilde{a}}  _{f .}$ is the $f$th row of ${\tilde{\bf A}}:=
{\bf A}^{ 2}{\bf A}^{*}$.

Denote by
  \begin{align*}u_{if}:={{\sum\limits_{\alpha \in I_{s,n} {\left\{ {f}
\right\}}} {{\rm{rdet}} _{f} \left( {\left( { {\bf A}^{ 3} \left({\bf A}^{ 3} \right)^{*}
} \right)_{. f} ({\check {\bf a}}_{i\,.})} \right) _{\alpha} ^{\alpha} } }
}.
\end{align*}
 Construct the matrix ${\bf U}=(u_{if})\in {\mathbb H}^{n\times n}$ and denote ${\tilde{\bf U}}={\bf U}{\tilde{\bf A}}={\bf U}{\bf A}^{ 2}{\bf A}^{*}$. It follows that
$$\sum\limits_{f}u_{if}
\sum\limits_{\alpha \in I_{s,n} {\left\{ {j}
\right\}}}{{\rm{rdet}} _{j} \left( {\left( { {\bf A}{\bf A}^{*}
} \right)_{j.} ( {\bf \tilde{a}}  _{f .})} \right)_{\alpha}
^{\alpha} }=\sum\limits_{\alpha \in I_{s,n} {\left\{ {j}
\right\}}}{{\rm{rdet}} _{j} \left( {\left( { {\bf A}{\bf A}^{*}
} \right)_{j.} ( {\bf \tilde{u}}  _{i .})} \right)_{\alpha}
^{\alpha} },
$$
where ${\bf \tilde{u}}  _{i .}$ is the $i$th row of ${\tilde{\bf U}}$.
Thus we have (\ref{eq:detrep_rcor}).
\end{proof}
Taking into account (\ref{eq:rep_lcor}), the following theorem on determinantal representations of the left core inverse can be proved similarly.

\begin{thm}\label{th:detrep_lcor}
Let ${\bf A}\in  {\mathbb H}_n^{\mathrm{CM}}$, $\rk{\bf A}^{2}=\rk{\bf A}=s$. Then its left core inverse   ${\bf A}_{\tiny\textcircled{\#}}=\left(a_{ij}^{\tiny\textcircled{\#},l}\right)$ has the following determinantal representation
 \begin{align}\label{eq:detrep_lcor}
&a_{ij}^{\tiny\textcircled{\#},l}= {\frac{\sum\limits_{\beta \in J_{s,n} {\left\{ {i}
\right\}}} {{{\rm{cdet}}_{i} {\left( {({\bf A}^{ *} {\bf A}
)_{.i} ({\bf \tilde{v}}_{.j} )}
\right)_{\beta} ^{\beta} } }}
}{{
{\sum\limits_{\beta \in J_{s,n}} {{ {\left| {
{\bf A}^{ *} {\bf A} } \right|_{\beta
}^{\beta} }  }}}
{\sum\limits_{\beta \in J_{s,n}} {{\left|  {\left({\bf A}^{3} \right)^{*} {\bf A}^{ 3}
}   \right|_{\alpha} ^{\alpha}}}} }}},
\end{align}
 where
${\bf \tilde{v}}_{.j} $ is the $j$th column of ${\tilde{\bf V}}:={\bf A}^{*}{\bf A}^{ 2}{\bf V}$, and ${\bf V}=(v_{fj})\in {\mathbb H}^{n\times n}$ such that
   \begin{align*}v_{fj}=\sum\limits_{\beta \in J_{s,n} {\left\{ {f}
\right\}}}{{\rm{cdet}} _{f} \left( {\left( { \left({\bf A}^{ 3} \right)^{*} {\bf A}^{ 3}
} \right)_{.f} ({\hat {\bf a}}_{.j})} \right)_{\beta} ^{\beta} },
\end{align*}
 where
 ${\widehat {\bf a}}_{.j}$ is the $j$th column of $\widehat{{\bf A}}:=
({\bf A}^{ 3})^{*}{\bf A}$.
\end{thm}

The next  corollary gives determinantal representations of the right and left core inverses for complex matrices.
 \begin{cor}
Let ${\bf A}\in {\mathbb C}_n^{\mathrm{CM}}$ and $\rk{\bf A}^2=\rk{\bf A}=s$. Then its right core inverse has the determinantal representations
\begin{align*}
&a_{ij}^{\tiny\textcircled{\#},r}=\frac{{\sum\limits_{\alpha \in I_{s,n} {\left\{ {j}
\right\}}} {\left| {( {\bf A}{\bf A}^{
*} )_{j.} ( {\bf {v}}_{i.})} \right|_{\alpha}
^{\alpha} } }}{\sum\limits_{\alpha \in I_{s,n}}\left| { {\bf A}{\bf A}^{*}  } \right|
_{\alpha} ^{\alpha}\sum\limits_{\beta \in J_{s,n}}\left| {{\bf A}^{2}  }
\right| _{\beta} ^{\beta}}=\frac{{{\sum\limits_{\beta
\in J_{s,n} {\left\{ {i} \right\}}} { \left|
{\left( { {\bf A}^{2}  } \right)_{. \,i}
\left( { {\bf v}_{.j} } \right) }\right|
 _{\beta} ^{\beta} } } }}{\sum\limits_{\alpha \in I_{s,n}}\left| { {\bf A}{\bf A}^{*}  } \right|
_{\alpha} ^{\alpha}\sum\limits_{\beta \in J_{s,n}}\left| {{\bf A}^{2}  }
\right| _{\beta} ^{\beta}},
\end{align*}
where
\begin{align*}
{\bf v}_{i.}=&\left[
\sum\limits_{\beta \in J_{s,n} {\left\{ {i} \right\}}}
{ \left| {\left( {{\bf A}^{
2}  } \right)_{.i} \left( {\bf{\tilde a}}_{.f}\right)}
\right|_{\beta} ^{\beta}} \right]\in {\mathbb{H}}^{1 \times
n},\,\,\,\,f=1,\ldots,n\\
    { {\bf v}_{.j} }=&\left[
{{{\sum\limits_{\alpha \in I_{s,n} {\left\{ {j}
\right\}}} { \left| {( {\bf A}{\bf A}^{
*}  )_{j.} ( {\bf\tilde a}_{l.})} \right|_{\alpha}
^{\alpha} } }}}
\right]\in {\mathbb{H}}^{n \times
1},\,\,\,\,l=1,\ldots,n,
\end{align*}
are the row-vector and the column-vector, respectively;  ${\bf \tilde{a}}  _{ .f}$ and ${\bf\tilde a}_{l.}$ are the $f$th column and  $l$th row of ${\tilde{\bf A}}:=
{\bf A}^{ 2}{\bf A}^{*}$.

And, its left core inverse has the determinantal representations
\begin{align*}
&a_{ij}^{\tiny\textcircled{\#},l}=\frac{{\sum\limits_{\alpha \in I_{s,n} {\left\{ {j}
\right\}}} {\left| {( {\bf A}^{
2} )_{j.} ( {\bf {u}}_{i.})} \right|_{\alpha}
^{\alpha} } }}{\sum\limits_{\alpha \in I_{s,n}}\left| { {\bf A}^{2}  } \right|
_{\alpha} ^{\alpha}\sum\limits_{\beta \in J_{s,n}}\left| {{\bf A}^{*}{\bf A}  }
\right| _{\beta} ^{\beta}}=\frac{{{\sum\limits_{\beta
\in J_{s,n} {\left\{ {i} \right\}}} { \left|
{\left( { {\bf A}^{*}{\bf A}  } \right)_{. \,i}
\left( { {\bf u}_{.j} }  \right|} \right)
 _{\beta} ^{\beta} } } }}{\sum\limits_{\alpha \in I_{s,n}}\left| { {\bf A}^{2}  } \right|
_{\alpha} ^{\alpha}\sum\limits_{\beta \in J_{s,n}}\left| {{\bf A}^{*}{\bf A}  }
\right| _{\beta} ^{\beta}},
\end{align*}
where
\begin{align*}
{\bf u}_{i.}=&\left[
\sum\limits_{\beta \in J_{s,n} {\left\{ {i} \right\}}}
{ \left| {\left( {{\bf A}^{
*}{\bf A}  } \right)_{.i} \left( {\bf{\bar a}}_{.f}\right)}
\right|_{\beta} ^{\beta}} \right]\in {\mathbb{H}}^{1 \times
n},\,\,\,\,f=1,\ldots,n\\
    { {\bf u}_{.j} }=&\left[
{{{\sum\limits_{\alpha \in I_{s,n} {\left\{ {j}
\right\}}} { \left| {( {\bf A}^{
2}  )_{j.} ( {\bf\bar a}_{l.})} \right|_{\alpha}
^{\alpha} } }}}
\right]\in {\mathbb{H}}^{n \times
1},\,\,\,\,l=1,\ldots,n,
\end{align*}
are the row-vector and the column-vector, respectively; ${\bf \bar{a}}  _{ .f}$ and ${\bf\bar a}_{l.}$ are the $f$th column and  $l$th row of ${\bar{\bf A}}:=
{\bf A}^{*}{\bf A}^{ 2}$.

\end{cor}

\subsection{Determinantal representations of the core EP inverses}
Similar as in \cite{pras}, we introduce two  core EP inverses.
\begin{defn}
A matrix ${\bf X}\in  {\mathbb{H}}^{n\times n}$ is called \emph{the right core EP inverse} of $ {\bf A}\in  {\mathbb{H}}^{n\times n}$
if it satisfies the conditions
$$
{\bf X}{\bf A}{\bf X}={\bf A},~and~ \mathcal{R}_r({\bf X})=\mathcal{R}_r({\bf X}^*)=\mathcal{R}_r({\bf A}^d).$$
It is denoted $ {\bf A}^{\tiny\textcircled{\dag}}$.
\end{defn}
\begin{defn}\label{def:lepcor}
A matrix ${\bf X}\in  {\mathbb{H}}^{n\times n}$ is called \emph{the left core EP inverse} of ${\bf A}\in  {\mathbb{H}}^{n\times n}$
if it satisfies the conditions
$$
{\bf X}{\bf A}{\bf X}={\bf A},~and~ \mathcal{R}_l({\bf X})=\mathcal{R}_l({\bf X}^*)=\mathcal{R}_l(({\bf A})^d).$$
It is denoted $ {\bf A}_{\tiny\textcircled{\dag}}$.
\end{defn}
\begin{rem}Since $\mathcal{R}_r(({\bf A}^*)^d)=\mathcal{R}_l(({\bf A})^d)$, then the left core inverse $ {\bf A}_{\tiny\textcircled{\dag}}$ of  $ {\bf A}\in{\mathbb{C}}^{n\times n}$ is similar to  the left $*$core inverse introduced in \cite{pras}, and  the dual core EP inverse introduced in \cite{zh_arx}.
\end{rem}
Due to \cite{pras}, we have the following representations the  core EP inverses of  $ {\bf A}\in{\mathbb{H}}^{n\times n}$ ,
 \begin{align*} {\bf A}^{\tiny\textcircled{\dag}}=&{\bf A}^{\{2,3,6a\}}~~\text{and}~~~\mathcal{R}_r({\bf A}^{\tiny\textcircled{\dag}})\subseteq \mathcal{R}_r({\bf A}^k),\\
{\bf A}_{\tiny\textcircled{\dag}}=&{\bf A}^{\{2,4,6b\}}~~\text{and}~~~\mathcal{R}_l({\bf A}_{\tiny\textcircled{\dag}})\subseteq \mathcal{R}_l({\bf A}^k).
\end{align*}
Thanks to \cite{{zh_arx}}, the following representations of the  core EP inverses  will be used for their determinantal representations.
\begin{lem}Let $ {\bf A}\in{\mathbb{H}}^{n\times n}$ and $Ind {\bf A}=k$. Then
 \begin{align} {\bf A}^{\tiny\textcircled{\dag}}=&{\bf A}^{k}({\bf A}^{k+1})^{\dag},\label{eq:rep_rcorep}\\
{\bf A}_{\tiny\textcircled{\dag}}=&({\bf A}^{k+1})^{\dag}{\bf A}^{k}.\label{eq:rep_lcorep}
\end{align}
Moreover, if $\Ind {\bf A}=1$, then we have the following representations of the right and left core inverses
\begin{align} {\bf A}^{\tiny\textcircled{\#}}=&{\bf A}({\bf A}^{2})^{\dag},\label{eq:rep_rcor_sim}\\
{\bf A}_{\tiny\textcircled{\#}}=&({\bf A}^{2})^{\dag}{\bf A}.\label{eq:rep_lcor_sim}
\end{align}
\end{lem}

\begin{thm}Suppose ${\bf A}\in  {\mathbb{H}}^{n\times n}_s$, $\Ind {\bf A}=k$, and there exist ${\bf A}^{\tiny\textcircled{\dag}}$ and ${\bf A}_{\tiny\textcircled{\dag}}$. Then ${\bf A}^{\tiny\textcircled{\dag}}=\left(a_{ij}^{\tiny\textcircled{\dag},r}\right)$ and ${\bf A}_{\tiny\textcircled{\dag}}=\left(a_{ij}^{\tiny\textcircled{\dag},l}\right)$ have the following determinantal representations, respectively,
 \begin{align}\label{eq:detrep_repcorep}
a_{ij}^{\tiny\textcircled{\dag},r}=&
 \frac{\sum\limits_{\alpha \in I_{s,n} {\left\{ {j}
\right\}}} {{\rm{rdet}} _{j} \left( {\left( { {\bf A}^{ k+1}\left({\bf A}^{k+1} \right)^{*}
} \right)_{j.} ({\hat {\bf a}}_{i\,.})} \right)_{\alpha} ^{\alpha} } }{{{\sum\limits_{\alpha \in I_{s,n}} {{\left|   {\bf A}^{k+1}\left({\bf A}^{k+1}\right)^{*}
  \right|_{\alpha} ^{\alpha}}}}
 }},\\\label{eq:detrep_lepcorep}
 a_{ij}^{\tiny\textcircled{\dag},l}=
& \frac{\sum\limits_{\beta \in J_{s,n} {\left\{ {i}
\right\}}} {{\rm{cdet}} _{i} \left( {\left( \left({\bf A}^{k+1} \right)^{*} {\bf A}^{ k+1}
\right)_{.i} ({\check {\bf a}}_{.j})} \right)_{\beta} ^{\beta} } }{{{\sum\limits_{\beta \in J_{s,n}} {{\left| \left({\bf A}^{k+1}\right)^{*} {\bf A}^{k+1}
  \right|_{\beta} ^{\beta}}}} }},
\end{align}
 where ${\hat {\bf a}}_{i\,.}$ is the $i$th row of $\hat{{\bf A}}=
{\bf A}^{ k}({\bf A}^{ k+1})^{*}$ and ${\check {\bf a}}_{.j}$ is the $j$th column of $\check{{\bf A}}=({\bf A}^{ k+1})^{*}
{\bf A}^{ k}$. 
\end{thm}
\begin{proof}Let $\left({\bf A}^{k+1}\right)^{\dag}=\left(a_{ij}^{(k+1)}\right)^{\dag}$ and  ${\bf A}^{k}=\left(a_{ij}^{(k)}\right)$.   By (\ref{eq:rep_rcorep}), $$a_{ij}^{\tiny\textcircled{\dag},r}=\sum\limits_{t=1}^na_{it}^{(k)}\left(a_{tj}^{(k+1)}\right)^{\dag}.$$
Using (\ref{eq:rdet_repr_AA*}) for the determinantal representation $\left({\bf A}^{k+1}\right)^{\dag}$, we obtain
 \begin{equation*}
a_{ij}^{\tiny\textcircled{\dag},r}=\sum\limits_{t=1}^na_{it}^{(k)}
{\frac{{{\sum\limits_{\alpha \in I_{s,n} {\left\{ {j} \right\}}}
{{\rm{rdet}} _{j} \left( {\left( {\bf A}^{k+1} ({\bf A}^{k+1})^{ *}
\right)_{j.} ( {\bf a}_{i.}^{(k+1,*)})} \right)_{\alpha}
^{\alpha} } }}}{{{\sum\limits_{\alpha \in I_{r,m}}  {{
{\left| { {\bf A}^{k+1} \left({\bf A}^{k+1}\right)^{ *} } \right|
_{\alpha} ^{\alpha} } }}} }}},
\end{equation*}
where ${\bf a}_{i.}^{(k+1,*)}$ is the $i$th row of $({\bf A}^{ k+1})^{*}$.
Since $\sum\limits_{t=1}^na_{it}^{(k)}{\bf a}_{\,i.}^{(k+1,*)})^{ }  ={\hat {\bf a}}_{i\,.}$, finally, we have (\ref{eq:detrep_repcorep}).

The determinantal representation (\ref{eq:detrep_lepcorep}) is obtained similarly by using (\ref{eq:cdet_repr_A*A}) for the determinantal representation $\left({\bf A}^{k+1}\right)^{\dag}$ in (\ref{eq:rep_lcorep}).
\end{proof}
Taking into account the representations (\ref{eq:rep_rcor_sim})-(\ref{eq:rep_lcor_sim}), we evidently obtain determinantal representations of the right and left core inverses which have more   simpler expressions than (\ref{eq:detrep_rcor})-(\ref{eq:detrep_lcor}).
\begin{cor}Let ${\bf A}\in  {\mathbb{H}}^{n\times n}_s$, $Ind {\bf A}=1$, and there exist ${\bf A}^{\tiny\textcircled{\#}}$ and ${\bf A}_{\tiny\textcircled{\#}}$. Then ${\bf A}^{\tiny\textcircled{\#}}=\left(a_{ij}^{\tiny\textcircled{\#},r}\right)$ and ${\bf A}_{\tiny\textcircled{\#}}=\left(a_{ij}^{\tiny\textcircled{\#},l}\right)$ have the following determinantal representations, respectively,
\begin{align}\label{eq:detrep_repcorep_sim}
a_{ij}^{\tiny\textcircled{\#},r}=&
 \frac{\sum\limits_{\alpha \in I_{s,n} {\left\{ {j}
\right\}}} {{\rm{rdet}} _{j} \left( {\left( { {\bf A}^{2}\left({\bf A}^{2} \right)^{*}
} \right)_{j.} ({\hat {\bf a}}_{i\,.})} \right)_{\alpha} ^{\alpha} } }{{{\sum\limits_{\alpha \in I_{s,n}} {{\left|   {\bf A}^{2}\left({\bf A}^{2}\right)^{*}
  \right|_{\alpha} ^{\alpha}}}}
 }},\\\label{eq:detrep_lepcorep_sim}
 a_{ij}^{\tiny\textcircled{\#},l}=
& \frac{\sum\limits_{\beta \in J_{s,n} {\left\{ {i}
\right\}}} {{\rm{cdet}} _{i} \left( {\left( \left({\bf A}^{2} \right)^{*} {\bf A}^{2}
\right)_{.i} ({\check {\bf a}}_{.j})} \right)_{\beta} ^{\beta} } }{{{\sum\limits_{\beta \in J_{s,n}} {{\left| \left({\bf A}^{2}\right)^{*} {\bf A}^{2}
  \right|_{\beta} ^{\beta}}}} }},
\end{align}
 where ${\hat {\bf a}}_{i\,.}$ is the $i$th row of $\hat{{\bf A}}=
{\bf A}({\bf A}^{2})^{*}$ and ${\check {\bf a}}_{.j}$ is the $j$th column of $\check{{\bf A}}=({\bf A}^{2})^{*}
{\bf A}$.
\end{cor}

The following corollary gives determinantal representations of the right and left  core EP inverses and the right and left  core inverses for complex matrices.
\begin{cor}Suppose ${\bf A}\in  {\mathbb{C}}^{n\times n}_s$, $Ind {\bf A}=k$, and there exist
 ${\bf A}^{\tiny\textcircled{\dag}}=\left(a_{ij}^{\tiny\textcircled{\dag},r}\right)$ and ${\bf A}_{\tiny\textcircled{\dag}}=\left(a_{ij}^{\tiny\textcircled{\dag},l}\right)$. Then they have the following determinantal representations, respectively,
 \begin{align*}
a_{ij}^{\tiny\textcircled{\dag},r}=&
 \frac{\sum\limits_{\alpha \in I_{s,n} {\left\{ {j}
\right\}}} { \left| {\left( { {\bf A}^{ k+1}\left({\bf A}^{k+1} \right)^{*}
} \right)_{j.} ({\hat {\bf a}}_{i\,.})} \right|_{\alpha} ^{\alpha} } }{{{\sum\limits_{\alpha \in I_{s,n}} {{\left|   {\bf A}^{k+1}\left({\bf A}^{k+1}\right)^{*}
  \right|_{\alpha} ^{\alpha}}}}
 }},\\
 a_{ij}^{\tiny\textcircled{\dag},l}=
& \frac{\sum\limits_{\beta \in J_{s,n} {\left\{ {i}
\right\}}} { \left| {\left( \left({\bf A}^{k+1} \right)^{*} {\bf A}^{ k+1}
\right)_{.i} ({\check {\bf a}}_{.j})} \right|_{\beta} ^{\beta} } }{{{\sum\limits_{\beta \in J_{s,n}} {{\left| \left({\bf A}^{k+1}\right)^{*} {\bf A}^{k+1}
  \right|_{\beta} ^{\beta}}}} }},
\end{align*}
 where ${\hat {\bf a}}_{i\,.}$ is the $i$th row of $\hat{{\bf A}}=
{\bf A}^{ k}({\bf A}^{ k+1})^{*}$ and ${\check {\bf a}}_{.j}$ is the $j$th column of $\check{{\bf A}}=({\bf A}^{ k+1})^{*}
{\bf A}^{ k}$.

If $Ind {\bf A}=1$, then
 ${\bf A}^{\tiny\textcircled{\#}}=\left(a_{ij}^{\tiny\textcircled{\#},r}\right)$ and ${\bf A}_{\tiny\textcircled{\#}}=\left(a_{ij}^{\tiny\textcircled{\#},l}\right)$ have the following determinantal representations, respectively,
 \begin{align*}
a_{ij}^{\tiny\textcircled{\#},r}=&
 \frac{\sum\limits_{\alpha \in I_{s,n} {\left\{ {j}
\right\}}} { \left| {\left( { {\bf A}^{ 2}\left({\bf A}^{2} \right)^{*}
} \right)_{j.} ({\hat {\bf a}}_{i\,.})} \right|_{\alpha} ^{\alpha} } }{{{\sum\limits_{\alpha \in I_{s,n}} {{\left|   {\bf A}^{2}\left({\bf A}^{2}\right)^{*}
  \right|_{\alpha} ^{\alpha}}}}
 }},\\
 a_{ij}^{\tiny\textcircled{\#},l}=
& \frac{\sum\limits_{\beta \in J_{s,n} {\left\{ {i}
\right\}}} { \left| {\left( \left({\bf A}^{2} \right)^{*} {\bf A}^{2}
\right)_{.i} ({\check {\bf a}}_{.j})} \right|_{\beta} ^{\beta} } }{{{\sum\limits_{\beta \in J_{s,n}} {{\left| \left({\bf A}^{2}\right)^{*} {\bf A}^{2}
  \right|_{\beta} ^{\beta}}}} }},
\end{align*}
 where ${\hat {\bf a}}_{i\,.}$ is the $i$th row of $\hat{\bf A}=
{\bf A}({\bf A}^{ 2})^{*}$ and ${\check {\bf a}}_{.j}$ is the $j$th column of $\check{{\bf A}}=({\bf A}^{2})^{*}
{\bf A}$.

\end{cor}

\subsection{Determinantal representations of the core DMP and MPD inverses}
 The concept of the DMP  inverse in complex matrices was introduced in \cite{mal1} by S. Malik and N. Thome that can be expended to quaternion matrices as follows.
\begin{defn}Suppose ${\bf A}\in  {\mathbb{H}}^{n\times n}$ and $\Ind {\bf A}=k$.
A matrix ${\bf X}\in  {\mathbb{H}}^{n\times n}$ is called \emph{the DMP inverse} of $ {\bf A}$
if it satisfies the conditions
 \begin{align}\label{eq:def_dmp}
{\bf X}{\bf A}{\bf X}={\bf X},~{\bf X}{\bf A}={\bf A}^d{\bf A},~and~ {\bf A}^k{\bf X}={\bf A}^k{\bf A}^{\dag}. \end{align}
It is denoted $ {\bf A}^{d,{\dag}}$.
\end{defn}It is proven \cite{mal1} that  the  matrix
satisfying system of equations (\ref{eq:def_dmp}) is unique and it has the following representation
\begin{align}\label{eq:rep_dmp}
{\bf A}^{d,{\dag}}={\bf A}^{d}{\bf A}{\bf A}^{\dag}. \end{align}
In accordance with the order of use the Drazin inverse (D) and the Moore-Penrose (MP) inverse, its name is the DMP inverse.
\begin{thm}\label{th:detrep_dmp}Let ${\bf A}\in  {\mathbb{H}}^{n\times n}_s$, $\Ind {\bf A}=k$, and $\rk ({\bf A}^{k})=s_1$. Then its DMP inverse $ {\bf A}^{d,{\dag}}= \left(a_{ij}^{d,{\dag}}\right)$ has the following determinantal representations.
\begin{enumerate}
  \item[(i)]~ If ${\bf A}$ is an arbitrary matrix, then
 \begin{align}\label{eq:detrep_dmp}
&a_{ij}^{d,{\dag}}= {\frac{
{{\sum\limits_{\alpha \in I_{s,n} {\left\{ {j}
\right\}}} {{{\rm{rdet}}_{j} {\left( {({\bf A} {\bf A}^{ *}
)_{j .} ({\bf \tilde{u}}_{i  .} )}
\right)_{\alpha} ^{\alpha} } }}}
}}{{{\sum\limits_{\alpha \in I_{s_1,n}} {{\left|  { {\bf A}^{2k+1} \left({\bf A}^{2k+1} \right)^{*}
}  \right|_{\alpha} ^{\alpha}}}}
{\sum\limits_{\alpha \in I_{s,n}} {{ {\left| {
{\bf A} {\bf A}^{ *} } \right|_{\alpha
}^{\alpha} }  }}}
 }}},
\end{align}
 where
${\bf \widetilde{u}}_{i .} $ is the $i$th row of ${\widetilde{\bf U}}:={\bf U}{\bf A}^{ 2}{\bf A}^{*}$, and ${\bf U}=(u_{if})\in {\mathbb H}^{n\times n}$ such that
   \begin{align*}u_{if}={{\sum\limits_{\alpha \in I_{s_1,n} {\left\{ {f}
\right\}}} {{\rm{rdet}} _{f} \left( {\left( { {\bf A}^{2k+1} \left({\bf A}^{ 2k+1} \right)^{*}
} \right)_{. f} ({\check {\bf a}}_{i\,.})} \right) _{\alpha} ^{\alpha} } }
},
\end{align*}
 where
 ${\check {\bf a}}_{i\,.}$ is the $i$th row of $\check{ {\bf A}}=
{\bf A}({\bf A}^{ 2k+1})^{*}$.
  \item[(ii)] If ${\bf A}$ is Hermitian, then
\begin{align}
 a_{ij}^{d,{\dag}}\label{eq:detrep_dmp_her_rdet}=&\frac{{\sum\limits_{\alpha \in I_{s,n} {\left\{ {j}
\right\}}} {{\rm{rdet}} _{j} \left( {( {\bf A}^{
2} )_{j.} ( {\bf {v}}_{i.})} \right)_{\alpha}
^{\alpha} } }}{\sum\limits_{\beta \in J_{s_1,n}} {{\left|
{ {{\bf A}^{k+1}}
}  \right|_{\beta} ^{\beta}}}
\sum\limits_{\alpha \in I_{s,n}}\left| { {\bf A}^{ 2}  } \right|
_{\alpha} ^{\alpha}}=\\
=&\frac{{{\sum\limits_{\beta
\in J_{s_1,n} {\left\{ {i} \right\}}} {{\rm{cdet}} _{i} \left(
{\left( { {\bf A}^{k+1}  } \right)_{. \,i}
\left( { {\bf u}_{.j} }  \right)} \right)
 _{\beta} ^{\beta} } } }}{\sum\limits_{\beta \in J_{s_1,n}} {{\left|
{ {{\bf A}^{k+1}}
}  \right|_{\beta} ^{\beta}}}\sum\limits_{\beta \in J_{s,n}}\left| {{\bf A}^{2}  }
\right| _{\beta} ^{\beta}}\label{eq:detrep_dmp_her_cdet},
\end{align}
where
\begin{align*}
{\bf v}_{i.}=&\left[
\sum\limits_{\beta \in J_{s_1,n} {\left\{ {i} \right\}}}
{{\rm{cdet}} _{i} \left( {\left( {{\bf A}^{
k+1}  } \right)_{.i} \left( {\bf{ a}}^{(k+2)}_{.f}\right)}
\right)_{\beta} ^{\beta}} \right]\in {\mathbb{H}}^{1 \times
n},\,\,\,\,f=1,\ldots,n\\
    { {\bf u}_{.j} }=&\left[
{{{\sum\limits_{\alpha \in I_{s,n} {\left\{ {j}
\right\}}} {{\rm{rdet}} _{j} \left( {( {\bf A}^{
2}  )_{j.} ( {\bf a}_{l.}^{(k+2)})} \right)_{\alpha}
^{\alpha} } }}}
\right]\in {\mathbb{H}}^{n \times
1},\,\,\,\,l=1,\ldots,n,
\end{align*}
are the row-vector and the column-vector, respectively.
\end{enumerate}
\end{thm}
\begin{proof}
By (\ref{eq:rep_dmp}),
\begin{equation*}
a_{ij}^{d,{\dag}} =  \sum\limits_{l = 1}^{n} {a}_{il}^{d} p_{lj}.
 \end{equation*}

(i) Let ${\bf A}\in  {\mathbb{H}}^{n\times n}_s$ be an arbitrary matrix. Then,
using the determinantal representations (\ref{eq:rdet_draz})  and (\ref{eq:det_repr_proj_P})  for respectively ${\bf A}^{d}$ and ${\bf P}_A={\bf A}{\bf A}^{\dag}$, we obtain
\begin{align*}
&a_{ij}^{d,{\dag}}= \sum\limits_{l = 1}^{n}{\frac{\sum\limits_{f = 1}^{n}{{\sum\limits_{\alpha \in I_{s_1,n} {\left\{ {f}
\right\}}} {{\rm{rdet}} _{f} \left( {\left( { {\bf A}^{ 2k+1} \left({\bf A}^{ 2k+1} \right)^{*}
} \right)_{. f} (\check{ {\bf a}}_{i\,.})} \right) _{\alpha} ^{\alpha} } }
}~{a}_{fl}}{{{\sum\limits_{\alpha \in I_{s_1,n}} {{\left|  { {\bf A}^{ 2k+1} \left({\bf A}^{2k+1} \right)^{*}
}  \right| _{\alpha} ^{\alpha}
}}} }}}\times\\
&~~~~~~~~~~~~~~~~~{\frac{{{\sum\limits_{\alpha \in I_{s,n} {\left\{ {j}
\right\}}} {{{\rm{rdet}} _{j} {\left( {({\bf A} {\bf A}^{ *}
)_{j .} ({\bf \ddot{a}}  _{l  .} )}
\right)  _{\alpha} ^{\alpha} } }}}
}}{{{\sum\limits_{\alpha \in I_{s,n}} {{ {\left| {
{\bf A} {\bf A}^{ *} } \right| _{\alpha
}^{\alpha} }  }}} }}}=\\
&{\frac{\sum\limits_{f = 1}^{n}{{\sum\limits_{\alpha \in I_{s_1,n} {\left\{ {f}
\right\}}} {{\rm{rdet}} _{f} \left( {\left( { {\bf A}^{ 2k+1} \left({\bf A}^{ 2k+1} \right)^{*}
} \right)_{. f} ({\check {\bf a}}_{i\,.})} \right) _{\alpha} ^{\alpha} } }
}~
{{\sum\limits_{\alpha \in I_{s,n} {\left\{ {j}
\right\}}} {{{\rm{rdet}} _{j} {\left( {({\bf A} {\bf A}^{ *}
)_{j .} ({\bf \tilde{a}}  _{f  .} )}
\right)  _{\alpha} ^{\alpha} } }}}
}}{{{\sum\limits_{\alpha \in I_{s_1,n}} {{\left| { { {\bf A}^{2k+1} \left({\bf A}^{2k+1} \right)^{*}
}
}  \right|_{\alpha} ^{\alpha}}}}
{\sum\limits_{\alpha \in I_{s,n}} {{ {\left| {
{\bf A} {\bf A}^{ *} } \right| _{\alpha
}^{\alpha} }  }}}
 }}},
\end{align*}
where ${\check {\bf a}}_{i\,.}$ is the $i$th row of $\check{ {\bf A}}=
{\bf A}({\bf A}^{ 2k+1})^{*}$ and ${\bf \widetilde{a}}  _{f .}$ is the $f$th row of ${\widetilde{\bf A}}:=
{\bf A}^{ 2}{\bf A}^{*}$.

Denote by
  \begin{align*}u_{if}:={{\sum\limits_{\alpha \in I_{s_1,n} {\left\{ {f}
\right\}}} {{\rm{rdet}} _{f} \left( {\left( { {\bf A}^{2k+1} \left({\bf A}^{2k+1} \right)^{*}
} \right)_{. f} ({\check {\bf a}}_{i\,.})} \right) _{\alpha} ^{\alpha} } }
}
\end{align*}
for all $i,f=1,\ldots,n$.
Now, we construct the matrix
 ${\bf U}=(u_{if})\in {\mathbb H}^{n\times n}$, and denote  ${\tilde{\bf U}}:={\bf U}{\tilde{\bf A}}={\bf U}{\bf A}^{ 2}{\bf A}^{*}$. It follows that
$$\sum\limits_{f}u_{if}\sum\limits_{\alpha \in I_{s,n} {\left\{ {j}
\right\}}}
{{\rm{rdet}} _{j} \left( {\left( { {\bf A}{\bf A}^{*}
} \right)_{j.} ( {\bf \widetilde{a}}  _{f .})} \right)_{\alpha}
^{\alpha} }=\sum\limits_{\alpha \in I_{s,n} {\left\{ {j}
\right\}}}{{\rm{rdet}} _{j} \left( {\left( { {\bf A}{\bf A}^{*}
} \right)_{j.} ( {\bf \widetilde{u}}  _{i .})} \right)_{\alpha}
^{\alpha} },
$$
where ${\bf \widetilde{u}}  _{i .}$ is the $i$th row of ${\tilde{\bf U}}$.
Thus we have (\ref{eq:detrep_dmp}).

(ii) Let, now,  ${\bf A}\in  {\mathbb{H}}^{n\times n}_s$  be Hermitian. Then using (\ref{eq:dr_rep_cdet} ) for the determinantal representation of ${\bf A}^{d}$ and (\ref{eq:det_rdet_mp_her}) for the determinantal representation of ${\bf A}^{\dag}$, we obtain
\begin{align*}
&a_{ij}^{d,{\dag}}=\nonumber\\& \sum\limits_{l = 1}^{n} \sum\limits_{f = 1}^{n}{\frac{{{\sum\limits_{\beta
\in J_{s_1,n} {\left\{ {i} \right\}}} {{\rm{cdet}} _{i} \left(
{\left( {{\bf A}^{k+1}} \right)_{. \,i} \left( {{\bf
a}_{.l}^{(k)}}  \right)} \right) _{\beta}
^{\beta} } } }}{{{\sum\limits_{\beta \in J_{s_1,n}} {{\left|
{ {{\bf A}^{k+1}}
}  \right|_{\beta} ^{\beta}}}} }}}
a_{lf}
{\frac{{{\sum\limits_{\alpha \in I_{s,n} {\left\{ {j} \right\}}}
{{\rm{rdet}} _{j} \left( {( {\bf A}^{ 2}
)_{j.} ( {\bf a}_{f.} )} \right)_{\alpha}
^{\alpha} } }}}{{{\sum\limits_{\alpha \in I_{s,n}}  {{
{\left| {{\bf A}^{2} } \right|
_{\alpha} ^{\alpha} } }}} }}}=\\
& \sum\limits_{l = 1}^{n} \sum\limits_{f = 1}^{n}{\frac{{{\sum\limits_{\beta
\in J_{s_1,n} {\left\{ {i} \right\}}} {{\rm{cdet}} _{i} \left(
{\left( {{\bf A}^{k+1}} \right)_{. \,i} \left( {{\bf
e}_{.l}}  \right)} \right) _{\beta}
^{\beta} } } }}{{{\sum\limits_{\beta \in J_{s_1,n}} {{\left|
{ {{\bf A}^{k+1}}
}  \right|_{\beta} ^{\beta}}}} }}}
a_{lf}^{(k+2)}
{\frac{{{\sum\limits_{\alpha \in I_{s,n} {\left\{ {j} \right\}}}
{{\rm{rdet}} _{j} \left( {( {\bf A}^{ 2}
)_{j.} ( {\bf e}_{f.} )} \right)_{\alpha}
^{\alpha} } }}}{{{\sum\limits_{\alpha \in I_{s,n}}  {{
{\left| {{\bf A}^{2} } \right|
_{\alpha} ^{\alpha} } }}} }}},
\end{align*}
where $ {\bf e}_{.l}$ and $ {\bf e}_{l.}$  are
the unit column-vector and  the unit row-vector, respectively, such that all their
components are $0$, except the $l$th components  which are $1$; $a_{lf}^{(k+2)}$ is the $(lf)$th element of the matrix ${\bf A}^{k+2}$.

If we denote by

\begin{align*}v_{if}:=&\sum\limits_{l = 1}^{n}{\sum\limits_{\beta
\in J_{s_1,n} {\left\{ {i} \right\}}} {{\rm{cdet}} _{i} \left(
{\left(  {\bf A}^{k+1}   \right)_{. i}
\left( { {\bf e}}_{.l}   \right)} \right)
 _{\beta} ^{\beta} } }{a}^{(k+2)}_{lf}=\\&{\sum\limits_{\beta
\in J_{s_1,n} {\left\{ {i} \right\}}} {{\rm{cdet}} _{i} \left(
{\left( { {\bf A}^{k+1} } \right)}_{.i}
\left(  {\bf {{a}}}_{.f}^{(k+2)}   \right) \right)
 _{\beta} ^{\beta} } }
\end{align*}
 the $f$th component of a row-vector ${\bf v}_{i.}=\left[v_{i1},\ldots,v_{in}\right]$, then
\begin{gather*}
\sum\limits_{f = 1}^{n}v_{if}{\sum\limits_{\alpha \in I_{s,n} {\left\{ {j}
\right\}}} {{\rm{rdet}} _{j} \left( {( {\bf A}^{
2} )_{j.} ( {\bf {e}}_{f.})} \right)_{\alpha}
^{\alpha} } }={\sum\limits_{\alpha \in I_{s,n} {\left\{ {j}
\right\}}} {{\rm{rdet}} _{j} \left( {( {\bf A}^{
2} )_{j.} ( {\bf {v}}_{i.})} \right)_{\alpha}
^{\alpha} } }.
\end{gather*}
So, we have (\ref{eq:detrep_dmp_her_rdet}).

If we denote by
  \begin{align*}u_{lj}:=&\sum\limits_{f = 1}^{n}{a}^{(k+2)}_{lf}
{{{\sum\limits_{\alpha \in I_{s,n} {\left\{ {j}
\right\}}} {{\rm{rdet}} _{j} \left( {( {\bf A}^{
2}  )_{j.} ( {\bf e}_{f.})} \right)_{\alpha}
^{\alpha} } }}}=\\&{{{\sum\limits_{\alpha \in I_{s,n} {\left\{ {j}
\right\}}} {{\rm{rdet}} _{j} \left( {( {\bf A}^{
2} )_{j.} ( {\bf a}_{l.}^{(k+2)})} \right)_{\alpha}
^{\alpha} } }}}
\end{align*}
 the $l$th component of a column-vector ${\bf u}_{.j}=\left[u_{1j},\ldots,u_{nj}\right]$, then
$$\sum\limits_{l = 1}^{n}
{{\sum\limits_{\beta
\in J_{s_1,n} {\left\{ {i} \right\}}} {{\rm{cdet}} _{i} \left(
{\left( { {\bf A}^{ k+1} } \right)_{. \,i}
\left( { {\bf e}_{.l} }  \right)} \right)
 _{\beta} ^{\beta} } } }
\,u_{lj}={{\sum\limits_{\beta
\in J_{s_1,n} {\left\{ {i} \right\}}} {{\rm{cdet}} _{i} \left(
{\left( { {\bf A}^{k+1}  } \right)_{. \,i}
\left( { {\bf u}_{.j} }  \right)} \right)
 _{\beta} ^{\beta} } } }.
$$
So, we get (\ref{eq:detrep_dmp_her_cdet}).
\end{proof}

In that connection, it would be logical to consider the following definition.
\begin{defn}\label{eq:def_mpd}
Suppose ${\bf A}\in  {\mathbb{H}}^{n\times n}$ and $\Ind {\bf A}=k$.
A matrix ${\bf X}\in  {\mathbb{H}}^{n\times n}$ is called \emph{the MPD inverse} of $ {\bf A}$
if it satisfies the conditions
 \begin{align*}
{\bf X}{\bf A}{\bf X}={\bf X},~{\bf A}{\bf X}={\bf A}{\bf A}^d,~and~ {\bf X}{\bf A}^k={\bf A}^{\dag}{\bf A}^k. \end{align*}
It is denoted $ {\bf A}^{{\dag},d}$.
\end{defn}
The  matrix $ {\bf A}^{{\dag},d}$ is unique, and it can be represented as
\begin{align}\label{eq:rep_mpd}
{\bf A}^{{\dag},d}={\bf A}^{\dag}{\bf A}{\bf A}^{d}. \end{align}
\begin{thm}\label{th:detrep_mpd}Let ${\bf A}\in  {\mathbb{H}}_s^{n\times n}$, $Ind {\bf A}=k$ and $\rk {\bf A}^k=s_1$. Then its MPD inverse $ {\bf A}^{{\dag},d}= \left(a_{ij}^{{\dag},d}\right)$ have the following determinantal representations.

(i)~ If ${\bf A}$ is an arbitrary matrix, then
 \begin{align*}
&a_{ij}^{{\dag},d}={\frac{{{\sum\limits_{\beta \in J_{s,n} {\left\{ {i}
\right\}}} {{{\rm{cdet}}_{i} {\left( {({\bf A}^{ *} {\bf A}
)_{.i} ({\bf \widetilde{v}}_{.j} )}
\right)_{\beta} ^{\beta} } }}}}
}{{{\sum\limits_{\beta \in J_{s,n}} {{ {\left| {
{\bf A}^{ *} {\bf A} } \right|_{\beta
}^{\beta} }  }}}
{\sum\limits_{\beta \in J_{s_1,n}} {{\left| { {\left({\bf A}^{2k+1} \right)^{*} {\bf A}^{2k+1}
}
}  \right|_{\beta} ^{\beta}}}}
}}},
\end{align*}
 where
${\bf \widetilde{v}}_{.j} $ is the $j$th column of ${\widetilde{\bf V}}:={\bf A}^{*}{\bf A}^{ 2}{\bf V}$, and ${\bf V}=(v_{fj})\in {\mathbb H}^{n\times n}$ such that
   \begin{align*}v_{fj}=\sum\limits_{\beta \in J_{s_1,n}\{f\}}{{\rm{cdet}} _{f} \left( {\left( { \left({\bf A}^{2k+1} \right)^{*} {\bf A}^{2k+1}
} \right)_{.f} ({\hat {\bf a}}_{.j})} \right)_{\beta} ^{\beta} },
\end{align*}
 where
 ${\widehat {\bf a}}_{.j}$ is the $j$th column of $\widehat{{\bf A}}=
({\bf A}^{ 2k+1})^{*}{\bf A}$.

(ii)~If ${\bf A}$ is Hermitian, then
\begin{align*}
a_{ij}^{{\dag},d}=&\frac{{\sum\limits_{\beta \in J_{s,n} {\left\{ {i}
\right\}}} {{\rm{cdet}} _{i} \left( {( {\bf A}^{
2} )_{.i} ( {\bf {v}}_{.j})} \right)_{\beta}
^\beta } }}{
\sum\limits_{\beta \in J_{s,n}}\left| { {\bf A}^{ 2}  } \right|
_{\beta} ^{\beta}
\sum\limits_{\beta \in I_{s_1,n}} {{\left|
{ {{\bf A}^{k+1}}
}  \right|_{\beta} ^{\beta}}}
}=\frac{{{\sum\limits_{\alpha
\in I_{s_1,n} {\left\{ {j} \right\}}} {{\rm{rdet}} _{j} \left(
{\left( { {\bf A}^{k+1}  } \right)_{j.}
\left( { {\bf u}_{i.} }  \right)} \right)
 _{\alpha} ^{\alpha} } } }}{\sum\limits_{\alpha \in I_{s,\,n}} {{\left|
{ {{\bf A}^{k+1}}
}  \right|_{\alpha} ^{\alpha}}}\sum\limits_{\alpha \in I_{s_1,n}}\left| {{\bf A}^{2}  }
\right| _{\alpha} ^{\alpha}},
\end{align*}

where
\begin{align*}
{\bf v}_{.j}=&\left[
\sum\limits_{\alpha\in I_{s_1,n} {\left\{ {j} \right\}}}
{{\rm{rdet}} _{j} \left( {\left( {{\bf A}^{
k+1}  } \right)_{j.} \left( {\bf{ a}}^{(k+2)}_{l.}\right)}
\right)_{\alpha} ^{\alpha}} \right]\in {\mathbb{H}}^{n \times
1},\,\,\,\,l=1,\ldots,n\\
    { {\bf u}_{i.} }=&\left[
{{{\sum\limits_{\beta \in J_{s,n} {\left\{ {i}
\right\}}} {{\rm{cdet}} _{i} \left( {( {\bf A}^{
2}  )_{.i} ( {\bf a}_{.f}^{(k+2)})} \right)_{\beta}
^{\beta} } }}}
\right]\in {\mathbb{H}}^{1 \times
n},\,\,\,\,l=1,\ldots,n.
\end{align*}
\end{thm}
\begin{proof}The proof is similar to the proof of Theorem \ref{th:detrep_dmp}.
By (\ref{eq:rep_mpd}),
\begin{equation}\label{eq:rep_mpd_com}
a_{ij}^{{\dag},d} =  \sum\limits_{l = 1}^{n} q_{il}{a}_{lj}^{d},
 \end{equation}

(i) If ${\bf A}\in  {\mathbb{H}}^{n\times n}_s$ is an arbitrary matrix,
 we use  (\ref{eq:det_repr_proj_Q}) for the determinantal representation of ${\bf Q}_A={\bf A}^{\dag}{\bf A}=(q_{ij})$ and
 (\ref{eq:cdet_draz}) for the determinantal representation of ${\bf A}^{d}$ in (\ref{eq:rep_mpd_com}).

 (ii) If ${\bf A}\in  {\mathbb{H}}^{n\times n}_s$ is Hermitian, then  we substitute in the equation  (\ref{eq:rep_mpd_com})    the determinantal representation (\ref{eq:det_cdet_mp_her}) for ${\bf A}^{\dag}$ and  for the determinantal representation (\ref{eq:dr_rep_rdet}) for ${\bf A}^{d}$ .
\end{proof}

The next  corollary gives determinantal representations of the DMP and MPD inverses for complex matrices.
 \begin{cor}
Let ${\bf A}\in  {\mathbb{C}}^{n\times n}_s$ with $\Ind {\bf A}=k$  and $\rk \left({\bf A}^{k}\right)=s_1$. Then its DMP inverse has determinantal representations
\begin{align*}
a_{ij}^{d,{\dag}}=\frac{{\sum\limits_{\alpha \in I_{s,n} {\left\{ {j}
\right\}}} { \left| {( {\bf A}{\bf A}^{
*} )_{j.} ( {\bf {v}}_{i.})} \right|_{\alpha}
^{\alpha} } }}{\sum\limits_{\beta \in J_{s_1,n}} {{\left|
{ {{\bf A}^{k+1}}
}  \right|_{\beta} ^{\beta}}}
\sum\limits_{\alpha \in I_{s,n}}\left| { {\bf A}{\bf A}^{*}  } \right|
_{\alpha} ^{\alpha}}=\frac{{{\sum\limits_{\beta
\in J_{s_1,n} {\left\{ {i} \right\}}} { \left|
{\left( { {\bf A}^{k+1}  } \right)_{. \,i}
\left( { {\bf u}_{.j} }  \right)} \right|
 _{\beta} ^{\beta} } } }}{\sum\limits_{\beta \in J_{s_1,n}} {{\left|
{ {{\bf A}^{k+1}}
}  \right|_{\beta} ^{\beta}}}\sum\limits_{\beta \in I_{s,n}}\left| {{\bf A}{\bf A}^{*}  }
\right| _{\alpha} ^{\alpha}},
\end{align*}
where
\begin{align*}
{\bf v}_{i.}=&\left[
\sum\limits_{\beta \in J_{s_1,n} {\left\{ {i} \right\}}}
{ \left| {\left( {{\bf A}^{
k+1}  } \right)_{.i} \left( {\bf{ a}}^{(k+2)}_{.f}\right)}
\right|_{\beta} ^{\beta}} \right]\in {\mathbb{C}}^{1 \times
n},\,\,\,\,f=1,\ldots,n\\
    { {\bf u}_{.j} }=&\left[
{{{\sum\limits_{\alpha \in I_{s,n} {\left\{ {j}
\right\}}} { \left| {( {\bf A}{\bf A}^{
*}  )_{j.} ( {\bf a}_{l.}^{(k+2)})} \right|_{\alpha}
^{\alpha} } }}}
\right]\in {\mathbb{C}}^{n \times
1},\,\,\,\,l=1,\ldots,n.
\end{align*}

And, its MPD inverse has determinantal representations
\begin{align*}
&a_{ij}^{{\dag},d}=\frac{{\sum\limits_{\beta \in J_{s,n} {\left\{ {i}
\right\}}} { \left| {( {\bf A}^{
*}{\bf A} )_{.i} ( {\bf {v}}_{.j})} \right|_{\beta}
^\beta } }}{
\sum\limits_{\beta \in J_{s,n}}\left| { {\bf A}^{*} {\bf A} } \right|
_{\beta} ^{\beta}
\sum\limits_{\alpha \in I_{s_1,n}} {{\left|
{ {{\bf A}^{k+1}}
}  \right|_{\alpha} ^{\alpha}}}
}=\frac{{{\sum\limits_{\alpha
\in I_{s_1,n} {\left\{ {j} \right\}}} { \left|
{\left( { {\bf A}^{k+1}  } \right)_{j.}
\left( { {\bf u}_{i.} }  \right)} \right|
 _{\alpha} ^{\alpha} } } }}{
\sum\limits_{\beta \in J_{s,n}}\left| { {\bf A}^{ *}{\bf A}  } \right|
_{\beta} ^{\beta}
\sum\limits_{\alpha \in I_{s_1,n}} {{\left|
{ {{\bf A}^{k+1}}
}  \right|_{\alpha} ^{\alpha}}}
},
\end{align*}
where
\begin{align*}
{\bf v}_{.j}=&\left[
\sum\limits_{\alpha\in I_{s_1,n} {\left\{ {j} \right\}}}
{ \left| {\left( {{\bf A}^{
k+1}  } \right)_{j.} \left( {\bf{ a}}^{(k+2)}_{l.}\right)}
\right|_{\alpha} ^{\alpha}} \right]\in {\mathbb{C}}^{n \times
1},\,\,\,\,l=1,\ldots,n\\
    { {\bf u}_{i.} }=&\left[
{{{\sum\limits_{\beta \in J_{s,n} {\left\{ {i}
\right\}}} { \left| {( {\bf A}^{
*}{\bf A}  )_{.i} ( {\bf a}_{.f}^{(k+2)})} \right|_{\beta}
^{\beta} } }}}
\right]\in {\mathbb{C}}^{1 \times
n},\,\,\,\,f=1,\ldots,n,.
\end{align*}
\end{cor}
\subsection{Determinantal representations of the  CMP inverse}
Recently, a new generalized inverse was investigated in \cite{meh} by M. Mehdipour and A. Salemi that can be extended to quaternion matrices as follows.
\begin{defn} Suppose ${\bf A}\in  {\mathbb{H}}^{n\times n}$  the core-nilpotent decomposition   ${\bf A}= {\bf A}_1 +{\bf A}_2$, where $\Ind{\bf A}_1=\Ind{\bf A}$, ${\bf A}_2$ is nilpotent and
${\bf A}_1{\bf A}_2={\bf A}_2{\bf A}_1=0$.
\emph{The CMP inverse} of ${\bf A}$ is called the matrix  ${\bf A}^{c,\dag}:={\bf A}^{\dag}{\bf A}_1{\bf A}^{\dag}$.
\end{defn}
Similarly to complex matrices  can be proved the next lemma.
\begin{lem}
Let ${\bf A}\in  {\mathbb{H}}^{n\times n}$. The matrix ${\bf X}={\bf A}^{c,\dag}$ is the unique matrix that
satisfies the following system of equations:
\begin{align*}
{\bf X}{\bf A}{\bf X}={\bf X},~{\bf A}{\bf X}{\bf A}={\bf A}_1,~{\bf A}{\bf X}={\bf A}_1{\bf A}^{\dag},~and~ {\bf X}{\bf A}={\bf A}^{\dag}{\bf A}_1. \end{align*}
Moreover,  \begin{align}\label{eq:rep_cmp}
{\bf A}^{c,\dag}={\bf Q}_A{\bf A}^d{\bf P}_A. \end{align}
\end{lem}
Taking into account (\ref{eq:rep_cmp}), it follows the next theorem  about determinantal representations of the quaternion CMP inverse.

\begin{thm}\label{th:detrep_cmp}Let ${\bf A}\in  {\mathbb{H}}^{n\times n}_s$,  $\Ind {\bf A}=m$, and $\rk \left({\bf A}^{m}\right)=s_1$. Then the  determinantal representations of its CMP inverse $ {\bf A}^{c,{\dag}}= \left(a_{ij}^{c,{\dag}}\right)$ can be expressed as

(i)~ when ${\bf A}$ is an arbitrary matrix
 \begin{align}
a_{ij}^{c,{\dag}}=&\label{eq:cdetrep_cmp} {\frac{{{\sum\limits_{\beta \in J_{s,n} {\left\{ {i}
\right\}}} {{{\rm{cdet}}_{i} {\left( {({\bf A}^{ *} {\bf A}
)_{.i} ({\bf {v}},^{(l)}_{.j} )}
\right)_{\beta} ^{\beta} } }}}}
}{{\left({\sum\limits_{\beta \in J_{s,n}} {{ {\left| {
{\bf A}^{ *} {\bf A} } \right|_{\beta
}^{\beta} }  }}}\right)^2
{\sum\limits_{\beta \in J_{s_1,n}} {{\left| {\left({\bf A}^{2m+1} \right)^{*} {\bf A}^{2m+1}
}   \right|_{\beta} ^{\beta}}}}
}}}=\\=&\label{eq:rdetrep_cmp}
 {\frac{\sum\limits_{\alpha \in I_{s,n} {\left\{ {j}
\right\}}}{{\rm{rdet}} _{j} \left( {\left( { {\bf A}{\bf A}^{*}
} \right)_{j.} ( {\bf w}\,^{(l)}_{i.})} \right)_{\alpha}
^{\alpha} }
}{{\left({\sum\limits_{\alpha \in I_{s,n}} {{ {\left| {
{\bf A} {\bf A}^{ *} } \right|_{\alpha
}^{\alpha} }  }}}\right)^2
{\sum\limits_{\beta \in J_{s_1,n}} {{\left|  {\left({\bf A}^{2m+1} \right)^{*} {\bf A}^{2m+1}
}   \right|_{\beta} ^{\beta}}}}
}}}
\end{align}
for all $l=1, 2$,
where

\begin{align}\label{eq:v1}
{\bf v}^{(1)}_{.j}=&\left[
{{{\sum\limits_{\alpha \in I_{s,n} {\left\{ {j}
\right\}}} {{\rm{rdet}} _{j} \left( {\left( { {\bf A}{\bf A}^{*}
} \right)_{j.} ( {\hat{\bf u}}_{t.})} \right)_{\alpha}
^{\alpha} } }}} \right]\in {\mathbb{H}}^{n \times
1},~t=1,\ldots,n,\\\label{eq:w1}
{\bf w}^{(1)}_{i.}=&\left[
{{{\sum\limits_{\beta \in J_{s,n} {\left\{ {i}
\right\}}} {{\rm{cdet}} _{i} \left( {\left( { {\bf A}^{*}{\bf A}
} \right)_{.i} ( {\hat{\bf u}}_{.k})} \right)_{\beta}
^{\beta} } }}} \right]\in {\mathbb{H}}^{1 \times
n},~k=1,\ldots,n,\\\label{eq:v2}
{\bf v}^{(2)}_{.j}=&\left[
{{{\sum\limits_{\alpha \in I_{s,n} {\left\{ {j}
\right\}}} {{\rm{rdet}} _{j} \left( {\left( { {\bf A}{\bf A}^{*}
} \right)_{j.} ( {\hat{\bf g}}_{t.})} \right)_{\alpha}
^{\alpha} } }}} \right]\in {\mathbb{H}}^{n \times
1},~t=1,\ldots,n,\\\label{eq:w2}
{\bf w}^{(2)}_{i.}=&\left[
{{{\sum\limits_{\beta \in J_{s,n} {\left\{ {i}
\right\}}} {{\rm{cdet}} _{i} \left( {\left( { {\bf A}^{*}{\bf A}
} \right)_{.i} ( {\hat{\bf g}}_{.k})} \right)_{\beta}
^{\beta} } }}} \right]\in {\mathbb{H}}^{1 \times
n},~k=1,\ldots,n.
\end{align}
Here
  ${\hat{\bf u}}_{t.}$ is the $t$th row and ${\hat{\bf u}}_{.k}$ is  the $k$th column of
$\hat{\bf U}:=
{\bf U}{\bf A}^{ m+1}{\bf A}^{*}$, ${\hat{\bf g}}_{t.}$ is the $t$th row and ${\hat{\bf g}}_{.k}$ is  the $k$th column of
$\hat{\bf G}:={\bf A}^{*}{\bf A}^{ 2}
{\bf G}$, and the matrices ${\bf U}=(u_{ij})\in{\mathbb{H}}^{n \times
n}$ and ${\bf G}=(g_{ij})\in{\mathbb{H}}^{n \times
n}$ are  such that 
\begin{align*}
 u_{ij}=&{\sum\limits_{\alpha \in I_{s_1,n} {\left\{ {j}
\right\}}} {{\rm{rdet}} _{j} \left( {\left( { {\bf A}^{ 2m+1} \left({\bf A}^{ 2m+1} \right)^{*}
} \right)_{j.} ( {\bf a}^{(1)}_{i.})} \right)_{\alpha}
^{\alpha} } },\\
g_{ij}=&{\sum\limits_{\beta \in J_{s_1,n} {\left\{ {i}
\right\}}} {{\rm{cdet}} _{i} \left( {\left(\left({\bf A}^{ 2m+1} \right)^{*}{\bf A}^{ 2m+1} \right)_{. i} \left({\bf a}^{(2)}_{.j}
\right)} \right)_{\beta} ^{\beta} } },
\end{align*}
 where ${\bf a}^{(1)}_{i.}$ is the $i$th row of $ {\bf A}_1:=
{\bf A}^*{\bf A}^{m+1}({\bf A}^{ 2m+1})^{*}$ and ${\bf a}^{(2)}_{.j}$ is the $j$th column of $ {\bf A}_2:=
({\bf A}^{ 2m+1})^*{\bf A}^{m+1}{\bf A}^{*}$.

(ii)~If ${\bf A}$ is Hermitian, then

 \begin{align}\label{eq:cdetrep_cmp_h}
&a_{ij}^{c,{\dag}}= {\frac{{{\sum\limits_{\beta \in J_{s,n} {\left\{ {i}
\right\}}} {{{\rm{cdet}}_{i} {\left( {({\bf A}^{ 2}
)_{.i} ({\bf {v}}^{(l)}_{.j} )}
\right)_{\beta} ^{\beta} } }}}}
}{{\left({\sum\limits_{\beta \in J_{s,n}} {{ {\left| {
{\bf A}^{2} } \right|_{\beta
}^{\beta} }  }}}\right)^2
{\sum\limits_{\beta \in J_{s_1,n}} {{\left| {\left( { {\bf A}^{m+1}
} \right)_{\beta} ^{\beta}
}  \right|}}}
}}}=\\\label{eq:rdetrep_cmp_h}
&a_{ij}^{c,{\dag}}= {\frac{\sum\limits_{\alpha \in I_{s,n} {\left\{ {j}
\right\}}}{{\rm{rdet}} _{j} \left( {\left( { {\bf A}^{2}
} \right)_{j.} ( {\bf w}^{(l)}_{i.})} \right)_{\alpha}
^{\alpha} }
}{{\left({\sum\limits_{\alpha \in I_{s,n}} {{ {\left| {
 {\bf A}^{ 2} } \right|_{\alpha
}^{\alpha} }  }}}\right)^2
{\sum\limits_{\beta \in J_{s_1,n}} {{\left| {\left( { {\bf A}^{m+1}
} \right)_{\beta} ^{\beta}
}  \right|}}}
}}}
\end{align}
for all $l=1, 2$,
where

\begin{align}\label{eq:v1_h}
{\bf v}^{(1)}_{.j}=&\left[
{{{\sum\limits_{\alpha \in I_{s,n} {\left\{ {j}
\right\}}} {{\rm{rdet}} _{j} \left( {\left( { {\bf A}^{2}
} \right)_{j.} ( {\hat{\bf u}}_{t.})} \right)_{\alpha}
^{\alpha} } }}} \right]\in {\mathbb{H}}^{n \times
1},~t=1,\ldots,n,\\\label{eq:w1_h}
{\bf w}^{(1)}_{i.}=&\left[
{{{\sum\limits_{\beta \in J_{s,n} {\left\{ {i}
\right\}}} {{\rm{cdet}} _{i} \left( {\left( { {\bf A}^{2}
} \right)_{.i} ( {\hat{\bf u}}_{.k})} \right)_{\beta}
^{\beta} } }}} \right]\in {\mathbb{H}}^{1 \times
n},~k=1,\ldots,n,\\\label{eq:v2_h}
{\bf v}^{(2)}_{.j}=&\left[
{{{\sum\limits_{\alpha \in I_{s,n} {\left\{ {j}
\right\}}} {{\rm{rdet}} _{j} \left( {\left( { {\bf A}^{2}
} \right)_{j.} ( {\hat{\bf g}}_{t.})} \right)_{\alpha}
^{\alpha} } }}} \right]\in {\mathbb{H}}^{n \times
1},~t=1,\ldots,n,\\\label{eq:w2_h}
{\bf w}^{(2)}_{i.}=&\left[
{{{\sum\limits_{\beta \in J_{s,n} {\left\{ {i}
\right\}}} {{\rm{cdet}} _{i} \left( {\left( { {\bf A}^{2}
} \right)_{.i} ( {\hat{\bf g}}_{.k})} \right)_{\beta}
^{\beta} } }}} \right]\in {\mathbb{H}}^{1 \times
n},~k=1,\ldots,n.
\end{align}
Here  ${\widehat{\bf u}}_{t.}$ is the $t$th row and ${\widehat{\bf u}}_{.k}$ is  the $k$th column of
$widehat{\bf U}:=
{\bf U}{\bf A}^{ 2}$, ${\hat{\bf g}}_{t.}$ is the $t$th row and ${\widehat{\bf g}}_{.k}$ is  the $k$th column of
$\widehat{\bf G}:={\bf A}^{ 2}
{\bf G}$, and the matrices ${\bf U}=(u_{ij})\in{\mathbb{H}}^{n \times
n}$ and ${\bf G}=(g_{ij})\in{\mathbb{H}}^{n \times
n}$ are  such that
\begin{align*}
 u_{ij}=&{\sum\limits_{\alpha \in I_{s_1,n} {\left\{ {j}
\right\}}} {{\rm{rdet}} _{j} \left( {\left( { {\bf A}^{ m+1}
} \right)_{j.} ( {\bf a}^{(m+2)}_{i.})} \right)_{\alpha}
^{\alpha} } },\\
g_{ij}=&{\sum\limits_{\beta \in J_{s_1,n} {\left\{ {i}
\right\}}} {{\rm{cdet}} _{i} \left( {\left({\bf A}^{ m+1} \right)_{. i} \left({\bf a}^{(m+2)}_{.j}
\right)} \right)_{\beta} ^{\beta} } }.
\end{align*}
\end{thm}
\begin{proof}
By (\ref{eq:rep_cmp}),
\begin{equation}\label{eq:cmp_start}
a_{ij}^{c,{\dag}} =  \sum_{l = 1}^n \sum_{k = 1}^n q_{il}{a}_{lk}^{d} p_{kj},
\end{equation}
where ${\bf Q}_A=(q_{il})$, ${\bf A}^d=(a^d_{il})$, and ${\bf P}_A=(p_{il})$.

(i) Let ${\bf A}\in  {\mathbb{H}}^{n\times n}_s$ be an arbitrary matrix.

a) Using  (\ref{eq:rdet_draz}), (\ref{eq:det_repr_proj_Q}), and (\ref{eq:det_repr_proj_P}) for the determinantal representations of   ${\bf A}^{d}$, ${\bf Q}_A={\bf A}^{\dag}{\bf A}$, and ${\bf P}_A={\bf A}{\bf A}^{\dag}$, respectively, we obtain

\begin{align*}
a_{ij}^{c,{\dag}} =&  \sum_{l = 1}^n \sum_{k = 1}^n {\frac{{{\sum\limits_{\beta \in J_{s,n} {\left\{ {i}
\right\}}} {{\rm{cdet}} _{i} \left( {\left( {{\bf A}^{ *}
{\bf A}} \right)_{.i} \left({\bf \dot{a}}_{.l} \right)}
\right)  _{\beta} ^{\beta} } }}}{{{\sum\limits_{\beta
\in J_{r,n}}  {{ \left| {{\bf A}^{ *}  {\bf
A}} \right|_{\beta}^{\beta} }}}} }}\times\\
&{\frac{\sum\limits_{t = 1}^{n}\left({{\sum\limits_{\alpha \in I_{s_1,n} {\left\{ {t}
\right\}}} {{\rm{rdet}} _{t} \left( {\left( { {\bf A}^{ 2m+1} \left({\bf A}^{ 2m+1} \right)^{*}
} \right)_{. t} (\check{ {\bf a}}_{l.})} \right) _{\alpha} ^{\alpha} } }
}\right){a}_{tk}^{(m)}}{{{\sum\limits_{\alpha \in I_{s_1,n}} {{\left|  { {\bf A}^{ 2m+1} \left({\bf A}^{2m+1} \right)^{*}
}  \right| _{\alpha} ^{\alpha}
}}} }}}\times\\&
{\frac{{{\sum\limits_{\alpha \in I_{s,n} {\left\{ {j}
\right\}}} {{{\rm{rdet}} _{j} {\left( {({\bf A} {\bf A}^{ *}
)_{j .} ({\bf \ddot{a}}  _{k  .} )}
\right)  _{\alpha} ^{\alpha} } }}}
}}{{{\sum\limits_{\alpha \in I_{s,n}} {{ {\left| {
{\bf A} {\bf A}^{ *} } \right| _{\alpha
}^{\alpha} }  }}} }}}=\end{align*}\begin{align*}
=& \sum_{l } \sum_{t } {\frac{{{\sum\limits_{\beta \in J_{s,n} {\left\{ {i}
\right\}}} {{\rm{cdet}} _{i} \left( {\left( {{\bf A}^{ *}
{\bf A}} \right)_{.i} \left({\bf \dot{a}}_{.l} \right)}
\right)  _{\beta} ^{\beta} } }}}{{{\sum\limits_{\beta
\in J_{s,n}}  {{ \left| {{\bf A}^{ *}  {\bf
A}} \right|_{\beta}^{\beta} }}}} }}\times\\&
{\frac{{{\sum\limits_{\alpha \in I_{s_1,n} {\left\{ {t}
\right\}}} {{\rm{rdet}} _{t} \left( {\left( { {\bf A}^{ 2m+1} \left({\bf A}^{ 2m+1} \right)^{*}
} \right)_{. t} (\check{ {\bf a}}_{l.})} \right) _{\alpha} ^{\alpha} } }
}}{{{\sum\limits_{\alpha \in I_{s_1,n}} {{\left|  { {\bf A}^{ 2m+1} \left({\bf A}^{2m+1} \right)^{*}
}  \right| _{\alpha} ^{\alpha}
}}} }}}\times\\
&{\frac{{{\sum\limits_{\alpha \in I_{s,n} {\left\{ {j}
\right\}}} {{{\rm{rdet}} _{j} {\left( {({\bf A} {\bf A}^{ *}
)_{j .} ({\bf \tilde{a}}  _{t .} )}
\right)  _{\alpha} ^{\alpha} } }}}
}}{{{\sum\limits_{\alpha \in I_{s,n}} {{ {\left| {
{\bf A} {\bf A}^{ *} } \right| _{\alpha
}^{\alpha} }  }}} }}},
\end{align*}
where ${\check {\bf a}}_{l.}$ is the $l$th row of $\check{ {\bf A}}:=
{\bf A}^m({\bf A}^{ 2m+1})^{*}$ and ${\bf \tilde{a}}  _{t .}$ is the $t$th row of ${\tilde{\bf A}}:=
{\bf A}^{m+1}{\bf A}^{*}$. Denote $ {\bf A}^*{\bf A}^{m+1}({\bf A}^{ 2m+1})^{*}=:{\bf A}_1=( a^{(1)}_{ij})$. So, it is clear that
\begin{align*}
a_{ij}^{c,{\dag}} = & \sum_{l } \sum_{t } \sum_{k }{\frac{{{\sum\limits_{\beta \in J_{s,n} {\left\{ {i}
\right\}}} {{\rm{cdet}} _{i} \left( {\left( {{\bf A}^{ *}
{\bf A}} \right)_{.i} \left({\bf {e}}_{.t} \right)}
\right)  _{\beta} ^{\beta} } }}}{{{\sum\limits_{\beta
\in J_{s,n}}  {{ \left| {{\bf A}^{ *}  {\bf
A}} \right|_{\beta}^{\beta} }}}} }}~a^{(1)}_{tk}\times
\\&{\frac{{{\sum\limits_{\alpha \in I_{s_1,n} {\left\{ {l}
\right\}}} {{\rm{rdet}} _{l} \left( {\left( { {\bf A}^{ 2m+1} \left({\bf A}^{ 2m+1} \right)^{*}
} \right)_{l.} ({ {\bf e}}_{k.})} \right) _{\alpha} ^{\alpha} } }
}}{{{\sum\limits_{\alpha \in I_{s,n}} {{\left|  { {\bf A}^{ 2m+1} \left({\bf A}^{2m+1} \right)^{*}
}  \right| _{\alpha} ^{\alpha}
}}} }}}\times\\
&{\frac{{{\sum\limits_{\alpha \in I_{s,n} {\left\{ {j}
\right\}}} {{{\rm{rdet}} _{j} {\left( {({\bf A} {\bf A}^{ *}
)_{j .} ({\bf \tilde{a}}  _{l  .} )}
\right)  _{\alpha} ^{\alpha} } }}}
}}{{{\sum\limits_{\alpha \in I_{s,n}} {{ {\left| {
{\bf A} {\bf A}^{ *} } \right| _{\alpha
}^{\alpha} }  }}} }}},
\end{align*}
where $ {\bf e}_{.t}$ and $ {\bf e}_{k.}$  are
the unit column and row vectors.

If we denote by
  \begin{align}\nonumber u_{tl}:=&\sum\limits_{k} a^{(1)}_{tk}
{{{\sum\limits_{\alpha \in I_{s_1,n} {\left\{ {j}
\right\}}} {{\rm{rdet}} _{l} \left( {\left( { {\bf A}^{ 2m+1} \left({\bf A}^{ 2m+1} \right)^{*}
} \right)_{l.} ( {\bf e}_{k.})} \right)_{\alpha}
^{\alpha} } }}}=\\&{{{\sum\limits_{\alpha \in I_{s_1,n} {\left\{ {j}
\right\}}} {{\rm{rdet}} _{l} \left( {\left( { {\bf A}^{ 2m+1} \left({\bf A}^{ 2m+1} \right)^{*}
} \right)_{l.} ( {\bf a}^{(1)}_{t.})} \right)_{\alpha}
^{\alpha} } }}}\label{eq:comp_u}
\end{align}
 the $t$th component of a column-vector ${\bf u}_{.\,l}=\left[u_{1l},\ldots,u_{nl}\right]$, then
$$\sum\limits_{t}
{{\sum\limits_{\beta
\in J_{s,n} {\left\{ {i} \right\}}} {{\rm{cdet}} \left( {\left( {{\bf A}^{ *}
{\bf A}} \right)_{.i}
\left( { {\bf e}_{.t} }  \right)} \right)
 _{\beta} ^{\beta} } } }
\,u_{tl}={{\sum\limits_{\beta
\in J_{s,n} {\left\{ {i} \right\}}} {{\rm{cdet}} _{i}  \left( {\left( {{\bf A}^{ *}
{\bf A}} \right)_{.i}
\left( { {\bf u}_{.l} }  \right)} \right)_{\beta} ^{\beta} } } }.
$$
It follows that
\begin{align*}
a_{ij}^{c,{\dag}} = & \sum_{l }{\frac{\sum\limits_{\beta \in J_{s,n} {\left\{ {i}
\right\}}} {{\rm{cdet}} _{i} \left( {\left( {{\bf A}^{ *}
{\bf A}} \right)_{.i} \left({\bf {u}}_{.l} \right)}
\right)  _{\beta} ^{\beta} }
{\sum\limits_{\alpha \in I_{s,n} {\left\{ {j}
\right\}}} {{{\rm{rdet}} _{j} {\left( {({\bf A} {\bf A}^{ *}
)_{j .} ({\bf \tilde{a}}  _{l .} )}
\right)  _{\alpha} ^{\alpha} } }}}
}{\sum\limits_{\beta
\in J_{s,n}}  {{ \left| {{\bf A}^{ *}  {\bf
A}} \right|_{\beta}^{\beta} }
{\sum\limits_{\alpha \in I_{s_1,n}} {{\left|  { {\bf A}^{ 2m+1} \left({\bf A}^{2m+1} \right)^{*}
}  \right| _{\alpha} ^{\alpha}
}}}{\sum\limits_{\alpha \in I_{s,n}} {{ {\left| {
{\bf A} {\bf A}^{ *} } \right| _{\alpha
}^{\alpha} }  }}}
} }}.
\end{align*}
Construct the matrix  ${\bf U}=\left(u_{tl}\right)\in  {\mathbb{H}}^{n \times n}$, where $u_{tl}$ is given by (\ref{eq:comp_u}).
Denote ${\bf U}{\tilde{\bf A}}=
{\bf U}{\bf A}^{m+1}{\bf A}^{*}=:\hat{\bf U}=\left({\hat u}_{tk}\right)$. Then, taking into account that $\left| {{\bf A}^{ *}  {\bf
A}} \right|_{\beta}^{\beta}=\left| {
{\bf A} {\bf A}^{ *} } \right| _{\alpha
}^{\alpha}$, we have

\begin{align*}
a_{ij}^{c,{\dag}} = &{\frac{ \sum_{t}\sum_{k}\sum\limits_{\beta \in J_{s,n} {\left\{ {i}
\right\}}} {{\rm{cdet}} _{i} \left( {\left( {{\bf A}^{ *}
{\bf A}} \right)_{.i} \left({\bf {e}}_{.t} \right)}
\right)  _{\beta} ^{\beta} }{\hat {u}}_{tk}
{\sum\limits_{\alpha \in I_{s,n} {\left\{ {j}
\right\}}} {{{\rm{rdet}} _{j} {\left( {({\bf A} {\bf A}^{ *}
)_{j .} ({\bf {e}}  _{k  .} )}
\right)  _{\alpha} ^{\alpha} } }}}
}{\sum\limits_{\beta
\in J_{s,n}}  {{ \left(\left| {{\bf A}^{ *}  {\bf
A}} \right|_{\beta}^{\beta}\right)^2 }
{\sum\limits_{\alpha \in I_{s_1,n}} {{\left|  { {\bf A}^{ 2m+1} \left({\bf A}^{2m+1} \right)^{*}
}  \right| _{\alpha} ^{\alpha}
}}}
} }}.
\end{align*}

If we denote by
  \begin{align*}v^{(1)}_{tj}:=\sum\limits_{k}{\hat u}_{tk}
{{{\sum\limits_{\alpha \in I_{s,n} {\left\{ {j}
\right\}}} {{\rm{rdet}} _{j} \left( {\left( { {\bf A}{\bf A}^{*}
} \right)_{j.} ( {\bf e}_{k.})} \right)_{\alpha}
^{\alpha} } }}}={{{\sum\limits_{\alpha \in I_{s,n} {\left\{ {j}
\right\}}} {{\rm{rdet}} _{j} \left( {\left( { {\bf A}{\bf A}^{*}
} \right)_{j.} ( {\hat{\bf u}}_{t.})} \right)_{\alpha}
^{\alpha} } }}}
\end{align*}
 the $t$th component of a column-vector ${\bf v}^{(1)}_{.j}=\left[v^{(1)}_{1j},\ldots,v^{(1)}_{nj}\right]$, then
$$\sum\limits_{t}
{{\sum\limits_{\beta
\in J_{s,n} {\left\{ {i} \right\}}} {{\rm{cdet}} \left( {\left( {{\bf A}^{ *}
{\bf A}} \right)_{.i}
\left( { {\bf e}_{.t} }  \right)} \right)
 _{\beta} ^{\beta} } } }
\,v^{(1)}_{tj}={{\sum\limits_{\beta
\in J_{s,n} {\left\{ {i} \right\}}} {{\rm{cdet}} _{i}  \left( {\left( {{\bf A}^{ *}
{\bf A}} \right)_{.i}
\left( { {\bf v}^{(1)}_{.j} }  \right)} \right)_{\beta} ^{\beta} } } }.
$$
Thus we have (\ref{eq:cdetrep_cmp}) with ${\bf v}^{(1)}_{.j}$ from (\ref{eq:v1}).

If we denote by
  \begin{align*} w^{(1)}_{ik}:=\sum\limits_{t}
{{{\sum\limits_{\beta \in J_{s,n} {\left\{ {i}
\right\}}} {{\rm{cdet}} _{i} \left( {\left( { {\bf A}^{*}{\bf A}
} \right)_{.i} ( {\bf e}_{.t})} \right)_{\beta}
^{\beta} } }}}{\hat u}_{tk}={{{\sum\limits_{\beta \in J_{s,n} {\left\{ {i}
\right\}}} {{\rm{cdet}} _{i} \left( {\left( { {\bf A}^{*}{\bf A}
} \right)_{.i} ( {\hat{\bf u}}_{.k})} \right)_{\beta}
^{\beta} } }}}
\end{align*}
 the $k$th component of a row-vector ${\bf w}^{(1)}_{i.}=\left[w^{(1)}_{i1},\ldots,w^{(1)}_{in}\right]$, then
$$\sum\limits_{k}w^{(1)}_{ik}\sum\limits_{\alpha \in I_{s,n} {\left\{ {j}
\right\}}}
{{\rm{rdet}} _{j} \left( {\left( { {\bf A}{\bf A}^{*}
} \right)_{j.} ( {\bf e}_{k.})} \right)_{\alpha}
^{\alpha} }=\sum\limits_{\alpha \in I_{s,n} {\left\{ {j}
\right\}}}{{\rm{rdet}} _{j} \left( {\left( { {\bf A}{\bf A}^{*}
} \right)_{j.} ( {\bf w}^{(1)}_{i.})} \right)_{\alpha}
^{\alpha} }.
$$
Thus we have (\ref{eq:rdetrep_cmp}) with ${\bf w}^{(1)}_{i.}$ from (\ref{eq:w1}).

b)
By
using the determinantal representation (\ref{eq:cdet_draz}) for  ${\bf A}^{d}$ in (\ref{eq:cmp_start}), we get
\begin{align*}
a_{ij}^{c,{\dag}} =&  \sum_{l = 1}^n \sum_{k = 1}^n {\frac{{{\sum\limits_{\beta \in J_{s,n} {\left\{ {i}
\right\}}} {{\rm{cdet}} _{i} \left( {\left( {{\bf A}^{ *}
{\bf A}} \right)_{.i} \left({\bf \dot{a}}_{.l} \right)}
\right)  _{\beta} ^{\beta} } }}}{{{\sum\limits_{\beta
\in J_{s,n}}  {{ \left| {{\bf A}^{ *}  {\bf
A}} \right|_{\beta}^{\beta} }}}} }}\times\\
& {\frac{{ \sum\limits_{t = 1}^{n} {a}_{lt}^{(m)}   {\sum\limits_{\beta \in J_{s_1,n} {\left\{ {t}
\right\}}} {{\rm{cdet}} _{t} \left( {\left(\left({\bf A}^{ 2m+1} \right)^{*}{\bf A}^{ 2m+1} \right)_{. t} \left( \tilde{ {\bf a}}_{.k}
\right)} \right)_{\beta} ^{\beta} } }
}}{{{\sum\limits_{\beta \in J_{s_1,n}} {{\left| {\left({\bf A}^{ 2m+1} \right)^{*}{\bf A}^{ 2m+1}
}  \right| _{\beta} ^{\beta}}}} }}}\times\\
&{\frac{{{\sum\limits_{\alpha \in I_{s,n} {\left\{ {j}
\right\}}} {{{\rm{rdet}} _{j} {\left( {({\bf A} {\bf A}^{ *}
)_{j .} ({\bf \ddot{a}}  _{k  .} )}
\right)  _{\alpha} ^{\alpha} } }}}
}}{{{\sum\limits_{\alpha \in I_{s,n}} {{ {\left| {
{\bf A} {\bf A}^{ *} } \right| _{\alpha
}^{\alpha} }  }}} }}}=\end{align*}\begin{align*}=&
 \sum_{k } \sum_{t } {\frac{{{\sum\limits_{\beta \in J_{s,n} {\left\{ {i}
\right\}}} {{\rm{cdet}} _{i} \left( {\left( {{\bf A}^{ *}
{\bf A}} \right)_{.i} \left({\bf \check{a}}_{.t} \right)}
\right)  _{\beta} ^{\beta} } }}}{{{\sum\limits_{\beta
\in J_{s,n}}  {{ \left| {{\bf A}^{ *}  {\bf
A}} \right|_{\beta}^{\beta} }}}} }} \times\\
&{\frac{{   {\sum\limits_{\beta \in J_{s_1,n} {\left\{ {t}
\right\}}} {{\rm{cdet}} _{t} \left( {\left(\left({\bf A}^{ 2m+1} \right)^{*}{\bf A}^{ 2m+1} \right)_{. t} \left(\tilde{ {\bf a}}_{.k}
\right)} \right)  _{\beta} ^{\beta} } }
}}{{{\sum\limits_{\beta \in J_{s_1,n}} {{\left| {\left({\bf A}^{ 2m+1} \right)^{*}{\bf A}^{ 2m+1}
}  \right|_{\beta} ^{\beta}}}} }}}\times\\
&{\frac{{{\sum\limits_{\alpha \in I_{s,n} {\left\{ {j}
\right\}}} {{{\rm{rdet}} _{j} {\left( {({\bf A} {\bf A}^{ *}
)_{j .} ({\bf \ddot{a}}  _{k .} )}
\right)  _{\alpha} ^{\alpha} } }}}
}}{{{\sum\limits_{\alpha \in I_{s,n}} {{ {\left| {
{\bf A} {\bf A}^{ *} } \right| _{\alpha
}^{\alpha} }  }}} }}},
\end{align*}
where ${\bf \check{a}}  _{.t}$ is the $t$th column of ${\check{\bf A}}:=
{\bf A}^{*}{\bf A}^{ 2}$  and ${\tilde {\bf a}}_{.k}$ is the $k$th column of $\tilde{ {\bf A}}=
({\bf A}^{ 2m+1})^{*}{\bf A}$.

Denote ${\bf A}_2=({ a}^{(2)}_{ij})=
({\bf A}^{ 2m+1})^{*}{\bf A}^2{\bf A}^*$. So, it is clear that
\begin{align*}
a_{ij}^{c,{\dag}} = & \sum_{l} \sum_{k } \sum_{t } {\frac{{{\sum\limits_{\beta \in J_{s,n} {\left\{ {i}
\right\}}} {{\rm{cdet}} _{i} \left( {\left( {{\bf A}^{ *}
{\bf A}} \right)_{.i} \left({\bf \check{a}}_{.t} \right)}
\right)  _{\beta} ^{\beta} } }}}{{{\sum\limits_{\beta
\in J_{s,n}}  {{ \left| {{\bf A}^{ *}  {\bf
A}} \right|_{\beta}^{\beta} }}}} }}\times\\
& {\frac{{   {\sum\limits_{\beta \in J_{s_1,n} {\left\{ {t}
\right\}}} {{\rm{cdet}} _{t} \left( {\left(\left({\bf A}^{ 2m+1} \right)^{*}{\bf A}^{ 2m+1} \right)_{. t} \left( { {\bf e}}_{.k}
\right)} \right)  _{\beta} ^{\beta} } }
}}{{{\sum\limits_{\beta \in J_{s_1,n}} {{\left| {\left(\left({\bf A}^{ 2m+1} \right)^{*}{\bf A}^{ 2m+1} \right) _{\beta} ^{\beta}
}  \right|}}} }}}~{a}^{(2)}_{kl}\times\\
&{\frac{{{\sum\limits_{\alpha \in I_{s,n} {\left\{ {j}
\right\}}} {{{\rm{rdet}} _{j} {\left( {({\bf A} {\bf A}^{ *}
)_{j .} ({\bf {e}}  _{l .} )}
\right)  _{\alpha} ^{\alpha} } }}}
}}{{{\sum\limits_{\alpha \in I_{s,n}} {{ {\left| {
{\bf A} {\bf A}^{ *} } \right| _{\alpha
}^{\alpha} }  }}} }}},
\end{align*}
where $ {\bf e}_{.k}$ and $ {\bf e}_{l.}$  are
the unit column and row vectors, respectively.

If we denote by
  \begin{align}\nonumber g_{tl}:=&\sum\limits_{l}{\sum\limits_{\beta \in J_{s_1,n} {\left\{ {t}
\right\}}} {{\rm{cdet}} _{t} \left( {\left(\left({\bf A}^{ 2m+1} \right)^{*}{\bf A}^{ 2m+1} \right)_{. t} \left( { {\bf e}}_{.k}
\right)} \right)_{\beta} ^{\beta} } }{ a}^{(2)}_{kl}
=\\\label{eq:comp_g}=&{\sum\limits_{\beta \in J_{s_1,n} {\left\{ {t}
\right\}}} {{\rm{cdet}} _{t} \left( {\left(\left({\bf A}^{ 2m+1} \right)^{*}{\bf A}^{ 2m+1} \right)_{. t} \left({\bf a}^{(2)}_{.l}
\right)} \right)_{\beta} ^{\beta} } }
\end{align}
 the $l$th component of a row-vector ${\bf g}_{t.}=\left[g_{t1},\ldots,g_{tn}\right]$, then
$$\sum\limits_{l}
g_{tl}{\sum\limits_{\alpha \in I_{s,n} {\left\{ {j}
\right\}}} {{{\rm{rdet}} _{j} {\left( {({\bf A} {\bf A}^{ *}
)_{j .} ({\bf {e}}  _{l .} )}
\right)  _{\alpha} ^{\alpha} } }}}={\sum\limits_{\alpha \in I_{s,n} {\left\{ {j}
\right\}}} {{{\rm{rdet}} _{j} {\left( {({\bf A} {\bf A}^{ *}
)_{j .} ({\bf {g}}  _{t .} )}
\right)  _{\alpha} ^{\alpha} } }}}.
$$
It follows that
\begin{align*}
a_{ij}^{c,{\dag}} = & \sum_{t }{\frac{\sum\limits_{\beta \in J_{s,n} {\left\{ {i}
\right\}}} {{\rm{cdet}} _{i} \left( {\left( {{\bf A}^{ *}
{\bf A}} \right)_{.i} \left({\bf {\check{a}}}_{.t} \right)}
\right)  _{\beta} ^{\beta} }
{\sum\limits_{\alpha \in I_{s,n} {\left\{ {j}
\right\}}} {{{\rm{rdet}} _{j} {\left( {({\bf A} {\bf A}^{ *}
)_{j .} ({\bf {g}}  _{t .} )}
\right)  _{\alpha} ^{\alpha} } }}}
}{\sum\limits_{\beta
\in J_{r,n}}  {{ \left| {{\bf A}^{ *}  {\bf
A}} \right|_{\beta}^{\beta} }
{\sum\limits_{\alpha \in I_{s_1,n}} {{\left|  { {\bf A}^{ 2m+1} \left({\bf A}^{2m+1} \right)^{*}
}  \right| _{\alpha} ^{\alpha}
}}}{\sum\limits_{\alpha \in I_{s,n}} {{ {\left| {
{\bf A} {\bf A}^{ *} } \right| _{\alpha
}^{\alpha} }  }}}
} }}.
\end{align*}
Construct the matrix  ${\bf G}=\left(g_{tl}\right)\in  {\mathbb{H}}^{n \times n}$, where $g_{tf}$ is given by (\ref{eq:comp_g}).
Denote ${\check{\bf A}}{\bf G}=
{\bf A}^{*}{\bf A}^{ 2}{\bf G}=:\hat{\bf G}=\left({\hat g}_{tl}\right)$. Then
\begin{align*}
a_{ij}^{c,{\dag}} = &{\frac{ \sum_{t}\sum_{k}\sum\limits_{\beta \in J_{s,n} {\left\{ {i}
\right\}}} {{\rm{cdet}} _{i} \left( {\left( {{\bf A}^{ *}
{\bf A}} \right)_{.i} \left({\bf {e}}_{.t} \right)}
\right)  _{\beta} ^{\beta} }{\hat {g}}_{tk}
{\sum\limits_{\alpha \in I_{s,n} {\left\{ {j}
\right\}}} {{{\rm{rdet}} _{j} {\left( {({\bf A} {\bf A}^{ *}
)_{j .} ({\bf {e}}  _{k  .} )}
\right)  _{\alpha} ^{\alpha} } }}}
}{\sum\limits_{\beta
\in J_{r,n}}  {{ \left(\left| {{\bf A}^{ *}  {\bf
A}} \right|_{\beta}^{\beta}\right)^2 }
{\sum\limits_{\alpha \in I_{s_1,n}} {{\left|  { {\bf A}^{ 2m+1} \left({\bf A}^{2m+1} \right)^{*}
}  \right| _{\alpha} ^{\alpha}
}}}
} }}.
\end{align*}

If we denote by
  \begin{align*}v^{(2)}_{tj}:=\sum\limits_{k}{\hat a}_{tk}
{{{\sum\limits_{\alpha \in I_{s,n} {\left\{ {j}
\right\}}} {{\rm{rdet}} _{j} \left( {\left( { {\bf A}{\bf A}^{*}
} \right)_{j.} ( {\bf e}_{k.})} \right)_{\alpha}
^{\alpha} } }}}={{{\sum\limits_{\alpha \in I_{s,n} {\left\{ {j}
\right\}}} {{\rm{rdet}} _{j} \left( {\left( { {\bf A}{\bf A}^{*}
} \right)_{j.} ( {\hat{\bf g}}_{t.})} \right)_{\alpha}
^{\alpha} } }}}
\end{align*}
 the $t$th component of a column-vector ${\bf v}^{(2)}_{.j}=\left[v^{(2)}_{1j},\ldots,v^{(2)}_{nj}\right]$, then
$$\sum\limits_{t}
{{\sum\limits_{\beta
\in J_{s,n} {\left\{ {i} \right\}}} {{\rm{cdet}} \left( {\left( {{\bf A}^{ *}
{\bf A}} \right)_{.i}
\left( { {\bf e}_{.t} }  \right)} \right)
 _{\beta} ^{\beta} } } }
\,v^{2}_{tj}={{\sum\limits_{\beta
\in J_{s,n} {\left\{ {i} \right\}}} {{\rm{cdet}} _{i}  \left( {\left( {{\bf A}^{ *}
{\bf A}} \right)_{.i}
\left( { {\bf v}^{(2)}_{.j} }  \right)} \right)_{\beta} ^{\beta} } } }.
$$
Thus we have (\ref{eq:cdetrep_cmp}) with ${\bf v}^{(2)}_{.j}$ from (\ref{eq:v2}).

If we denote by
  \begin{align*} w^{(2)}_{ik}:=\sum\limits_{t}
{{{\sum\limits_{\beta \in J_{s,n} {\left\{ {i}
\right\}}} {{\rm{cdet}} _{i} \left( {\left( { {\bf A}^{*}{\bf A}
} \right)_{.i} ( {\bf e}_{.t})} \right)_{\beta}
^{\beta} } }}}{\hat g}_{tk}={{{\sum\limits_{\beta \in J_{s,n} {\left\{ {i}
\right\}}} {{\rm{cdet}} _{i} \left( {\left( { {\bf A}^{*}{\bf A}
} \right)_{.i} ( {\hat{\bf g}}_{.k})} \right)_{\beta}
^{\beta} } }}}
\end{align*}
 the $k$th component of a row-vector ${\bf w}^{(2)}_{i.}=\left[w^{(2)}_{i1},\ldots,w^{(2)}_{in}\right]$, then
$$\sum\limits_{k}w^{(2)}_{ik}\sum\limits_{\alpha \in I_{s,n} {\left\{ {j}
\right\}}}
{{\rm{rdet}} _{j} \left( {\left( { {\bf A}{\bf A}^{*}
} \right)_{j.} ( {\bf e}_{k.})} \right)_{\alpha}
^{\alpha} }=\sum\limits_{\alpha \in I_{s,n} {\left\{ {j}
\right\}}}{{\rm{rdet}} _{j} \left( {\left( { {\bf A}{\bf A}^{*}
} \right)_{j.} ( {\bf w}^{(2)}_{i.})} \right)_{\alpha}
^{\alpha} }.
$$
Thus we have (\ref{eq:rdetrep_cmp}) with ${\bf w}^{(2)}_{i.}$ from (\ref{eq:w2}).

(ii) Let ${\bf A}\in  {\mathbb{H}}^{n\times n}_s$ be  Hermitian.

a) By
using the determinantal representations (\ref{eq:dr_rep_rdet}) for  ${\bf A}^{d}$, (\ref{eq:det_repr_proj_Q}) for  ${\bf Q}_A={\bf A}^{\dag}{\bf A}$, and (\ref{eq:det_repr_proj_P}) for ${\bf P}_A={\bf A}{\bf A}^{\dag}$, and taking into account Hermicity of ${\bf A}$,  we have
\begin{align*}
a_{ij}^{c,{\dag}} = &
 \sum_{l } \sum_{t } {\frac{{{\sum\limits_{\beta \in J_{s,n} {\left\{ {i}
\right\}}} {{\rm{cdet}} _{i} \left( {\left( {{\bf A}^{2}} \right)_{.i} \left( {\bf a}^{(2)}_{.t} \right)}
\right)  _{\beta} ^{\beta} } }}}{{{\sum\limits_{\beta
\in J_{s,n}}  {{ \left| {{\bf A}^{2}} \right|_{\beta}^{\beta} }}}} }}\times\\
&{\frac{{{\sum\limits_{\alpha
\in I_{s_1,n} {\left\{ {l} \right\}}} {{\rm{rdet}} _{l} \left(
{({\bf A}^{ m+1} )_{l.} \left( {\bf a}_{t.}^{ (m)} \right)}
\right)_{\alpha} ^{\alpha} } }}}{{{\sum\limits_{\alpha \in
I_{s_1,n}}  {{\left| {\left( { {\bf A}^{m+1} } \right){\kern
1pt}  _{\alpha} ^{\alpha} } \right|}}} }}}
{\frac{{{\sum\limits_{\alpha \in I_{s,n} {\left\{ {j}
\right\}}} {{{\rm{rdet}} _{j} {\left( {({\bf A}^{ 2}
)_{j .} ({\bf a}^{(2)}  _{l .} )}
\right)  _{\alpha} ^{\alpha} } }}}
}}{{{\sum\limits_{\alpha \in I_{s,n}} {{ {\left| {
{\bf A}^{ 2} } \right| _{\alpha
}^{\alpha} }  }}} }}},
\end{align*}
where ${\bf a}^{(2)}_{.t}$ and ${\bf a}^{(2)}  _{l .}$ are the $t$th column and the $l$th row of $ {\bf A}^{2}$, and ${\bf a}_{t.}^{ (m)}$ is the $t$th row of $
{\bf A}^{ m}$.  So, it is clear that

\begin{align*}
a_{ij}^{c,{\dag}} = & \sum_{l } \sum_{t } \sum_{k }{\frac{{{\sum\limits_{\beta \in J_{s,n} {\left\{ {i}
\right\}}} {{\rm{cdet}} _{i} \left( {\left( {{\bf A}^{2}} \right)_{.i} \left({\bf {e}}_{.t} \right)}
\right)  _{\beta} ^{\beta} } }}}{{{\sum\limits_{\beta
\in J_{s,n}}  {{ \left| {{\bf A}^{2}} \right|_{\beta}^{\beta} }}}} }}~a^{(m+2)}_{tk}\times\\
&{\frac{{{\sum\limits_{\alpha \in I_{s_1,n} {\left\{ {l}
\right\}}} {{\rm{rdet}} _{l} \left( {\left( { {\bf A}^{m+1}
} \right)_{l.} ({ {\bf e}}_{k.})} \right) _{\alpha} ^{\alpha} } }
}}{{{\sum\limits_{\alpha \in I_{s_1,n}} {{\left|  { {\bf A}^{ m+1}
}  \right| _{\alpha} ^{\alpha}
}}} }}}{\frac{{{\sum\limits_{\alpha \in I_{s,n} {\left\{ {j}
\right\}}} {{{\rm{rdet}} _{j} {\left( {({\bf A}^ {2}
)_{j .} ({\bf {a}} ^{(2)} _{l  .} )}
\right)  _{\alpha} ^{\alpha} } }}}
}}{{{\sum\limits_{\alpha \in I_{s,n}} {{ {\left| {
 {\bf A}^{2} } \right| _{\alpha
}^{\alpha} }  }}} }}},
\end{align*}
where $ {\bf e}_{.t}$ and $ {\bf e}_{k.}$  are
the unit column and row vectors, respectively.

If we denote by
  \begin{align}\nonumber u_{tl}:=&\sum\limits_{k} a^{(m+2)}_{tk}
{{{\sum\limits_{\alpha \in I_{s_1,n} {\left\{ {j}
\right\}}} {{\rm{rdet}} _{l} \left( {\left( { {\bf A}^{ m+1}
} \right)_{l.} ( {\bf e}_{k.})} \right)_{\alpha}
^{\alpha} } }}}=\\\label{eq:comp_u_h}=&{{{\sum\limits_{\alpha \in I_{s_1,n} {\left\{ {j}
\right\}}} {{\rm{rdet}} _{l} \left( {\left( { {\bf A}^{ m+1}
} \right)_{l.} ( {\bf a}^{(m+2)}_{t.})} \right)_{\alpha}
^{\alpha} } }}}
\end{align}
 the $t$th component of a column-vector ${\bf u}_{.\,l}=\left[u_{1l},\ldots,u_{nl}\right]$, then
$$\sum\limits_{t}
{{\sum\limits_{\beta
\in J_{s,n} {\left\{ {i} \right\}}} {{\rm{cdet}} \left( {\left( {{\bf A}^{2}} \right)_{.i}
\left( { {\bf e}_{.t} }  \right)} \right)
 _{\beta} ^{\beta} } } }
\,u_{tl}={{\sum\limits_{\beta
\in J_{s,n} {\left\{ {i} \right\}}} {{\rm{cdet}} _{i}  \left( {\left( {{\bf A}^{2}} \right)_{.i}
\left( { {\bf u}_{.l} }  \right)} \right)_{\beta} ^{\beta} } } }.
$$
It follows that
\begin{align*}
a_{ij}^{c,{\dag}} = & \sum_{l }{\frac{\sum\limits_{\beta \in J_{s,n} {\left\{ {i}
\right\}}} {{\rm{cdet}} _{i} \left( {\left( {{\bf A}^{ 2}
} \right)_{.i} \left({\bf {u}}_{.l} \right)}
\right)  _{\beta} ^{\beta} }
{\sum\limits_{\alpha \in I_{s,n} {\left\{ {j}
\right\}}} {{{\rm{rdet}} _{j} {\left( {( {\bf A}^{2}
)_{j .} ({\bf {a}}^{(2)}  _{l .} )}
\right)  _{\alpha} ^{\alpha} } }}}
}{\sum\limits_{\beta
\in J_{r,n}}  {{ \left| {{\bf A}^{2}  } \right|_{\beta}^{\beta} }
{\sum\limits_{\alpha \in I_{s_1,n}} {{\left|  { {\bf A}^{ m+1}
}  \right| _{\alpha} ^{\alpha}
}}}{\sum\limits_{\alpha \in I_{s,n}} {{ {\left| {
 {\bf A}^{2} } \right| _{\alpha
}^{\alpha} }  }}}
} }}.
\end{align*}
Construct the matrix  ${\bf U}=\left(u_{tl}\right)\in  {\mathbb{H}}^{n \times n}$, where $u_{tl}$ is given by (\ref{eq:comp_u_h}).
Denote $\hat{\bf U}:=
{\bf U}{\bf A}^{ 2}$. Then, taking into account that $\left| {\bf A}^{ 2}\right|_{\beta}^{\beta}=\left|  {\bf A}^{ 2} \right| _{\alpha
}^{\alpha}$, we have

\begin{align*}
a_{ij}^{c,{\dag}} = &{\frac{ \sum_{t}\sum_{k}\sum\limits_{\beta \in J_{s,n} {\left\{ {i}
\right\}}} {{\rm{cdet}} _{i} \left( {\left( {\bf A}^{ 2}
 \right)_{.i} \left({\bf {e}}_{.t} \right)}
\right)  _{\beta} ^{\beta} }{\hat {u}}_{tk}
{\sum\limits_{\alpha \in I_{s,n} {\left\{ {j}
\right\}}} {{{\rm{rdet}} _{j} {\left( {( {\bf A}^{2}
)_{j .} ({\bf {e}}  _{k  .} )}
\right)  _{\alpha} ^{\alpha} } }}}
}{\sum\limits_{\beta
\in J_{s,n}}  {{ \left(\left| {\bf A}^{2} \right|_{\beta}^{\beta}\right)^2 }
{\sum\limits_{\alpha \in I_{s_1,n}} {{\left|  { {\bf A}^{m+1}
}  \right| _{\alpha} ^{\alpha}
}}}
} }}.
\end{align*}

If we denote by
  \begin{align*}v^{(1)}_{tj}:=&\sum\limits_{k}{\hat u}_{tk}
{{{\sum\limits_{\alpha \in I_{s,n} {\left\{ {j}
\right\}}} {{\rm{rdet}} _{j} \left( {\left( {\bf A}^{2}
 \right)_{j.} ( {\bf e}_{k.})} \right)_{\alpha}
^{\alpha} } }}}=\\&{{{\sum\limits_{\alpha \in I_{s,n} {\left\{ {j}
\right\}}} {{\rm{rdet}} _{j} \left( {\left( { {\bf A}^{2}
} \right)_{j.} ( {\hat{\bf u}}_{t.})} \right)_{\alpha}
^{\alpha} } }}}
\end{align*}
 the $t$th component of a column-vector ${\bf v}^{(1)}_{.j}=\left[v^{(1)}_{1j},\ldots,v^{(1)}_{nj}\right]$, then
$$\sum\limits_{t}
{{\sum\limits_{\beta
\in J_{s,n} {\left\{ {i} \right\}}} {{\rm{cdet}} \left( {\left( {\bf A}^{2}
 \right)_{.i}
\left( { {\bf e}_{.t} }  \right)} \right)
 _{\beta} ^{\beta} } } }
\,v^{(1)}_{tj}={{\sum\limits_{\beta
\in J_{s,n} {\left\{ {i} \right\}}} {{\rm{cdet}} _{i}  \left( {\left( {\bf A}^{2}
 \right)_{.i}
\left( { {\bf v}^{(1)}_{.j} }  \right)} \right)_{\beta} ^{\beta} } } }.
$$
Thus we have (\ref{eq:cdetrep_cmp_h}) with ${\bf v}^{1}_{.j}$ from (\ref{eq:v1_h}).

If we denote by
  \begin{align*} w^{(1)}_{ik}:=&\sum\limits_{t}
{{{\sum\limits_{\beta \in J_{s,n} {\left\{ {i}
\right\}}} {{\rm{cdet}} _{i} \left( {\left(  {\bf A}^{2} \right)_{.i} ( {\bf e}_{.t})} \right)_{\beta}
^{\beta} } }}}{\hat u}_{tk}=\\&{{{\sum\limits_{\beta \in J_{s,n} {\left\{ {i}
\right\}}} {{\rm{cdet}} _{i} \left( {\left(  {\bf A}^{2} \right)_{.i} ( {\hat{\bf u}}_{.k})} \right)_{\beta}
^{\beta} } }}}
\end{align*}
 the $k$th component of a row-vector ${\bf w}^{(1)}_{i.}=\left[w^{(1)}_{i1},\ldots,w^{(1)}_{in}\right]$, then
$$\sum\limits_{k}w^{(1)}_{ik}\sum\limits_{\alpha \in I_{s,n} {\left\{ {j}
\right\}}}
{{\rm{rdet}} _{j} \left( {\left( {\bf A}^{2}
 \right)_{j.} ( {\bf e}_{k.})} \right)_{\alpha}
^{\alpha} }=\sum\limits_{\alpha \in I_{s,n} {\left\{ {j}
\right\}}}{{\rm{rdet}} _{j} \left( {\left( {\bf A}^{2}
 \right)_{j.} ( {\bf w}^{(1)}_{i.})} \right)_{\alpha}
^{\alpha} }.
$$
Thus we have (\ref{eq:rdetrep_cmp_h}) with ${\bf w}^{(1)}_{i.}$ from (\ref{eq:w1_h}).

b)
By
using the determinantal representation (\ref{eq:dr_rep_cdet}) for  ${\bf A}^{d}$ in (\ref{eq:cmp_start}), we get
\begin{align*}
a_{ij}^{c,{\dag}} =
& \sum_{k } \sum_{t } {\frac{{{\sum\limits_{\beta \in J_{s,n} {\left\{ {i}
\right\}}} {{\rm{cdet}} _{i} \left( {\left( {{\bf A}^{ 2}
} \right)_{.i} \left({\bf a}^{(2)}_{.t} \right)}
\right)  _{\beta} ^{\beta} } }}}{{{\sum\limits_{\beta
\in J_{s,n}}  {{ \left| {{\bf A}^{2} } \right|_{\beta}^{\beta} }}}} }}\times\\
& {\frac{{   {\sum\limits_{\beta \in J_{s_1,n} {\left\{ {t}
\right\}}} {{\rm{cdet}} _{t} \left( {\left({\bf A}^{ m+1} \right)_{. t} \left( {\bf a}^{(m)}_{.k}
\right)} \right)  _{\beta} ^{\beta} } }
}}{{{\sum\limits_{\beta \in J_{s,n}} {{\left| {\left({\bf A}^{ m+1} \right) _{\beta} ^{\beta}
}  \right|}}} }}}{\frac{{{\sum\limits_{\alpha \in I_{s,n} {\left\{ {j}
\right\}}} {{{\rm{rdet}} _{j} {\left( {( {\bf A}^{2}
)_{j .} ({\bf a}^{(2)}  _{k .} )}
\right)  _{\alpha} ^{\alpha} } }}}
}}{{{\sum\limits_{\alpha \in I_{s,n}} {{ {\left| {
 {\bf A}^{2} } \right| _{\alpha
}^{\alpha} }  }}} }}}.
\end{align*}

It is clear that
\begin{align*}
a_{ij}^{c,{\dag}} = & \sum_{l} \sum_{k } \sum_{t } {\frac{{{\sum\limits_{\beta \in J_{s,n} {\left\{ {i}
\right\}}} {{\rm{cdet}} _{i} \left( {\left( {\bf A}^{ 2}
 \right)_{.i} \left({\bf a}^{(2)}_{.t} \right)}
\right)  _{\beta} ^{\beta} } }}}{{{\sum\limits_{\beta
\in J_{s,n}}  {{ \left| {{\bf A}^{ 2} } \right|_{\beta}^{\beta} }}}} }}\times\\
& {\frac{{   {\sum\limits_{\beta \in J_{s_1,n} {\left\{ {t}
\right\}}} {{\rm{cdet}} _{t} \left( {\left({\bf A}^{ m+1} \right)_{. t} \left( { {\bf e}}_{.k}
\right)} \right)  _{\beta} ^{\beta} } }
}}{{{\sum\limits_{\beta \in J_{s_1,n}} {{\left| {\left({\bf A}^{ m+1} \right) _{\beta} ^{\beta}
}  \right|}}} }}}~{a}^{(m+2)}_{kl}~{\frac{{{\sum\limits_{\alpha \in I_{s,n} {\left\{ {j}
\right\}}} {{{\rm{rdet}} _{j} {\left( {({\bf A}^{ 2}
)_{j .} ({\bf {e}}  _{l .} )}
\right)  _{\alpha} ^{\alpha} } }}}
}}{{{\sum\limits_{\alpha \in I_{s,n}} {{ {\left| {
 {\bf A}^{ 2} } \right| _{\alpha
}^{\alpha} }  }}} }}}.
\end{align*}

If we denote by
  \begin{align}\nonumber g_{tl}:=&\sum\limits_{l}{\sum\limits_{\beta \in J_{s_1,n} {\left\{ {t}
\right\}}} {{\rm{cdet}} _{t} \left( {\left({\bf A}^{ m+1} \right)_{. t} \left( { {\bf e}}_{.k}
\right)} \right)_{\beta} ^{\beta} } }{ a}^{(m+2)}_{kl}
=\\&{\sum\limits_{\beta \in J_{s_1,n} {\left\{ {t}
\right\}}} {{\rm{cdet}} _{t} \left( {\left({\bf A}^{ m+1} \right)_{. t} \left({\bf a}^{(m+2)}_{.l}
\right)} \right)_{\beta} ^{\beta} } }\label{eq:comp_g_h}
\end{align}
 the $l$th component of a row-vector ${\bf g}_{t.}=\left[g_{t1},\ldots,g_{tn}\right]$, then
$$\sum\limits_{l}
g_{tl}{\sum\limits_{\alpha \in I_{s,n} {\left\{ {j}
\right\}}} {{{\rm{rdet}} _{j} {\left( {( {\bf A}^{ 2}
)_{j .} ({\bf {e}}  _{l .} )}
\right)  _{\alpha} ^{\alpha} } }}}={\sum\limits_{\alpha \in I_{s,n} {\left\{ {j}
\right\}}} {{{\rm{rdet}} _{j} {\left( {( {\bf A}^{2}
)_{j .} ({\bf {g}}  _{t .} )}
\right)  _{\alpha} ^{\alpha} } }}}.
$$
It follows that
\begin{align*}
a_{ij}^{c,{\dag}} = & \sum_{t }{\frac{\sum\limits_{\beta \in J_{s,n} {\left\{ {i}
\right\}}} {{\rm{cdet}} _{i} \left( {\left( {{\bf A}^{ 2}
} \right)_{.i} \left({\bf a}^{(2)}_{.t} \right)}
\right)  _{\beta} ^{\beta} }
{\sum\limits_{\alpha \in I_{s,n} {\left\{ {j}
\right\}}} {{{\rm{rdet}} _{j} {\left( {( {\bf A}^{ 2}
)_{j .} ({\bf {g}}  _{t .} )}
\right)  _{\alpha} ^{\alpha} } }}}
}{\sum\limits_{\beta
\in J_{r,n}}  {{ \left| {\bf A}^{ 2} \right|_{\beta}^{\beta} }
{\sum\limits_{\alpha \in I_{s_1,n}} {{\left|  { {\bf A}^{ m+1}
}  \right| _{\alpha} ^{\alpha}
}}}{\sum\limits_{\alpha \in I_{s,n}} {{ {\left| {
 {\bf A}^{ 2} } \right| _{\alpha
}^{\alpha} }  }}}
} }}.
\end{align*}
Construct the matrix ${\bf G}=\left(g_{tl}\right)\in  {\mathbb{H}}^{n \times n}$, where $g_{tl}$ is given by (\ref{eq:comp_g_h}).
Denote $\hat{\bf G}:={\check{\bf A}}{\bf G}=
{\bf A}^{ 2}{\bf G}$. Then
\begin{align*}
a_{ij}^{c,{\dag}} = &{\frac{ \sum_{t}\sum_{k}\sum\limits_{\beta \in J_{s,n} {\left\{ {i}
\right\}}} {{\rm{cdet}} _{i} \left( {\left( {{\bf A}^{2}
} \right)_{.i} \left({\bf {e}}_{.t} \right)}
\right)  _{\beta} ^{\beta} }{\hat {g}}_{tk}
{\sum\limits_{\alpha \in I_{s,n} {\left\{ {j}
\right\}}} {{{\rm{rdet}} _{j} {\left( {( {\bf A}^{2}
)_{j .} ({\bf {e}}  _{k  .} )}
\right)  _{\alpha} ^{\alpha} } }}}
}{\sum\limits_{\beta
\in J_{s,n}}  {{ \left(\left| {\bf A}^{2} \right|_{\beta}^{\beta}\right)^2 }
{\sum\limits_{\alpha \in I_{s_1,n}} {{\left|  { {\bf A}^{ m+1}
}  \right| _{\alpha} ^{\alpha}
}}}
} }}.
\end{align*}
If we denote by
  \begin{align*}v^{(2)}_{tj}:=&\sum\limits_{k}{\hat a}_{tk}
{{{\sum\limits_{\alpha \in I_{s,n} {\left\{ {j}
\right\}}} {{\rm{rdet}} _{j} \left( {\left( { {\bf A}^{2}
} \right)_{j.} ( {\bf e}_{k.})} \right)_{\alpha}
^{\alpha} } }}}=&\\{{{\sum\limits_{\alpha \in I_{s,n} {\left\{ {j}
\right\}}} {{\rm{rdet}} _{j} \left( {\left( {\bf A}^{2}
 \right)_{j.} ( {\hat{\bf g}}_{t.})} \right)_{\alpha}
^{\alpha} } }}}
\end{align*}
 the $t$th component of a column-vector ${\bf v}^{(2)}_{.j}=\left[v^{(2)}_{1j},\ldots,v^{(2)}_{nj}\right]$, then
$$\sum\limits_{t}
{{\sum\limits_{\beta
\in J_{s,n} {\left\{ {i} \right\}}} {{\rm{cdet}} \left( {\left( {\bf A}^{ 2}
 \right)_{.i}
\left( { {\bf e}_{.t} }  \right)} \right)
 _{\beta} ^{\beta} } } }
\,v^{(2)}_{tj}={{\sum\limits_{\beta
\in J_{s,n} {\left\{ {i} \right\}}} {{\rm{cdet}} _{i}  \left( {\left( {{\bf A}^{2}
} \right)_{.i}
\left( { {\bf v}^{(2)}_{.j} }  \right)} \right)_{\beta} ^{\beta} } } }.
$$
Thus we have (\ref{eq:cdetrep_cmp_h}) with ${\bf v}^{2}_{.j}$ from (\ref{eq:v2_h}).

If we denote by
  \begin{align*} w^{(2)}_{ik}:=&\sum\limits_{t}
{{{\sum\limits_{\beta \in J_{s,n} {\left\{ {i}
\right\}}} {{\rm{cdet}} _{i} \left( {\left( { {\bf A}^{2}
} \right)_{.i} ( {\bf e}_{.t})} \right)_{\beta}
^{\beta} } }}}{\hat g}_{tk}=\\&{{{\sum\limits_{\beta \in J_{s,n} {\left\{ {i}
\right\}}} {{\rm{cdet}} _{i} \left( {\left( { {\bf A}^{2}
} \right)_{.i} ( {\hat{\bf g}}_{.k})} \right)_{\beta}
^{\beta} } }}}
\end{align*}
 the $k$th component of a row-vector ${\bf w}^{(2)}_{i.}=\left[w^{(2)}_{i1},\ldots,w^{(2)}_{in}\right]$, then
$$\sum\limits_{k}w^{(2)}_{ik}\sum\limits_{\alpha \in I_{s,n} {\left\{ {j}
\right\}}}
{{\rm{rdet}} _{j} \left( {\left( {\bf A}^{2}
 \right)_{j.} ( {\bf e}_{k.})} \right)_{\alpha}
^{\alpha} }=\sum\limits_{\alpha \in I_{s,n} {\left\{ {j}
\right\}}}{{\rm{rdet}} _{j} \left( {\left( { {\bf A}^{2}
} \right)_{j.} ( {\bf w}^{(2)}_{i.})} \right)_{\alpha}
^{\alpha} }.
$$
Thus we have (\ref{eq:rdetrep_cmp_h}) with ${\bf w}^{(2)}_{i.}$ from (\ref{eq:w2_h}).
\end{proof}

The following  corollary gives determinantal representations of the CMP inverse for complex matrices.
 \begin{cor}
\label{th:detrep_c}Let ${\bf A}\in  {\mathbb{C}}^{n\times n}_s$, $\Ind {\bf A}=m$ and $\rk {\bf A}^m=s_1$. Then its CMP inverse $ {\bf A}^{c,{\dag}}= \left(a_{ij}^{c,{\dag}}\right)$ has the following determinantal representations

 \begin{align*}
a_{ij}^{c,{\dag}}=& {\frac{{{\sum\limits_{\beta \in J_{s,n} {\left\{ {i}
\right\}}} {{ {\left| {({\bf A}^{ *} {\bf A}
)_{.i} ({\bf {v}}^{(l)}_{.j} )}
\right|_{\beta} ^{\beta} } }}}}
}{{\left({\sum\limits_{\beta \in J_{s,n}} {{ {\left| {
{\bf A}^{ *} {\bf A} } \right|_{\beta
}^{\beta} }  }}}\right)^2
{\sum\limits_{\beta \in J_{s_1,n}} {{\left| {\left( { {\bf A}^{m+1}
} \right)_{\beta} ^{\beta}
}  \right|}}}
}}}=\\&= {\frac{\sum\limits_{\alpha \in I_{s,n} {\left\{ {j}
\right\}}}{ \left| {\left( { {\bf A}{\bf A}^{*}
} \right)_{j.} ( {\bf w}^{(l)}_{i.})} \right|_{\alpha}
^{\alpha} }
}{{\left({\sum\limits_{\alpha \in I_{s,n}} {{ {\left| {
{\bf A} {\bf A}^{ *} } \right|_{\alpha
}^{\alpha} }  }}}\right)^2
{\sum\limits_{\beta \in J_{s_1,n}} {{\left| {\left( { {\bf A}^{m+1}
} \right)_{\beta} ^{\beta}
}  \right|}}}
}}}
\end{align*}
for all $l=1, 2$,
where
\begin{align*}
{\bf v}^{(1)}_{.j}=&\left[
{{{\sum\limits_{\alpha \in I_{s,n} {\left\{ {j}
\right\}}} {\left| {\left( { {\bf A}{\bf A}^{*}
} \right)_{j.} ( {\hat{\bf u}}_{t.})} \right|_{\alpha}
^{\alpha} } }}} \right]\in {\mathbb{C}}^{n \times
1},~t=1,\ldots,n,\\
{\bf w}^{(1)}_{i.}=&\left[
{{{\sum\limits_{\beta \in J_{s,n} {\left\{ {i}
\right\}}} { \left| {\left( { {\bf A}^{*}{\bf A}
} \right)_{.i} ( {\hat{\bf u}}_{.k})} \right|_{\beta}
^{\beta} } }}} \right]\in {\mathbb{C}}^{1 \times
n},~k=1,\ldots,n,\\
{\bf v}^{(2)}_{.j}=&\left[
{{{\sum\limits_{\alpha \in I_{s,n} {\left\{ {j}
\right\}}} { \left| {\left( { {\bf A}{\bf A}^{*}
} \right)_{j.} ( {\hat{\bf g}}_{t.})} \right|_{\alpha}
^{\alpha} } }}} \right]\in {\mathbb{C}}^{n \times
1},~t=1,\ldots,n,\\
{\bf w}^{(2)}_{i.}=&\left[
{{{\sum\limits_{\beta \in J_{s,n} {\left\{ {i}
\right\}}} { \left| {\left( { {\bf A}^{*}{\bf A}
} \right)_{.i} ( {\hat{\bf g}}_{.k})} \right|_{\beta}
^{\beta} } }}} \right]\in {\mathbb{C}}^{1 \times
n},~k=1,\ldots,n,
\end{align*}
Here
  ${\hat{\bf u}}_{t.}$ is the $t$th row and ${\hat{\bf u}}_{.k}$ is  the $k$th column of
$\hat{\bf U}=
{\bf U}{\bf A}^{ 2}$, ${\hat{\bf g}}_{t.}$ is the $t$th row and ${\hat{\bf g}}_{.k}$ is  the $k$th column of
$\hat{\bf G}={\bf A}^{ 2}
{\bf G}$, and the matrices ${\bf U}=(u_{ij})\in{\mathbb{C}}^{n \times
n}$ and ${\bf G}=(g_{ij})\in{\mathbb{C}}^{n \times
n}$ are  such that
\begin{align*}
 u_{ij}=&{\sum\limits_{\alpha \in I_{s_1,n} {\left\{ {j}
\right\}}} { \left| {\left( { {\bf A}^{ m+1}
} \right)_{j.} ( {\bf a}^{(m+2)}_{i.})} \right|_{\alpha}
^{\alpha} } },\\
g_{ij}=&{\sum\limits_{\beta \in J_{s_1,n} {\left\{ {i}
\right\}}} { \left| {\left({\bf A}^{ m+1} \right)_{. i} \left({\bf a}^{(m+2)}_{.j}
\right)} \right|_{\beta} ^{\beta} } }.
\end{align*}

\end{cor}

\section{An example}
Given the matrix
\begin{align*}
{\bf A}=
      \begin{bmatrix}
      0& \mathbf{k}&0 \\
        \mathbf{i} &0&-\mathbf{j} \\
       0&\mathbf{j}&0\\
      \end{bmatrix}.
\end{align*}
Since
\begin{align*}
{\bf A}^*{\bf A}=
      \begin{bmatrix}
      1& 0&\mathbf{k} \\
        0 &2&0 \\
       \mathbf{k}&0&1\\
      \end{bmatrix},
\end{align*}then (determinantal) $\rk({\bf A}^*{\bf A})=\rk({\bf A})=2$. Similarly, it can be find \begin{align*}
{\bf A}^2=&
      \begin{bmatrix}
      \mathbf{j}& 0&\mathbf{i} \\
        0 &1-\mathbf{j}&0 \\
       -\mathbf{k}&0&1\\
      \end{bmatrix},~{\bf A}^2\left({\bf A}^2\right)^*=
      \begin{bmatrix}
      2& 0&2\mathbf{i} \\
        0 &2&0 \\
       -2\mathbf{i}&0&2\\
      \end{bmatrix},
\end{align*}
and $\rk({\bf A}^2)=2$. So, $\Ind {\bf A}=1$ and we shall find ${\bf A}^{\tiny\textcircled{\#}}$ and ${\bf A}_{\tiny\textcircled{\#}}$ by (\ref{eq:detrep_repcorep_sim}) and (\ref{eq:detrep_lepcorep_sim}), respectively.

Since \begin{align*}
\hat{{\bf A}}={\bf A}({\bf A}^2)^*=
      \begin{bmatrix}
      0& -\mathbf{i}+\mathbf{k}&0 \\
        -2\mathbf{k} &0&-2\mathbf{j} \\
       0&-1+\mathbf{j}&0\\
      \end{bmatrix},\end{align*} then  by (\ref{eq:detrep_repcorep_sim})
 \begin{align*}a_{11}^{\tiny\textcircled{\#},r}=&
 \frac{\sum\limits_{\alpha \in I_{2,3} {\left\{ {1}
\right\}}} {{\rm{rdet}} _{1} \left( {\left( { {\bf A}^{2}\left({\bf A}^{2} \right)^{*}
} \right)_{1.} ({\hat {\bf a}}_{1.})} \right)_{\alpha} ^{\alpha} } }{{{\sum\limits_{\alpha \in I_{2,3}} {{\left|   {\bf A}^{2}\left({\bf A}^{2}\right)^{*}
  \right|_{\alpha} ^{\alpha}}}}
 }}=\\&\frac{1}{8}\left({\rm{rdet}} _{1}\begin{bmatrix}
         0& -\mathbf{i}+\mathbf{k} \\
          0 &2
      \end{bmatrix}+{\rm{rdet}} _{1}\begin{bmatrix}
         0& 0 \\
          -2i &2
      \end{bmatrix}\right)=0.\end{align*}
      Similarly,
\begin{align*}
a_{21}^{\tiny\textcircled{\#},r}=&
=\frac{1}{8}\left({\rm{rdet}} _{1}\begin{bmatrix}
         -2\mathbf{k}& 0 \\
          0 &2
      \end{bmatrix}+{\rm{rdet}} _{1}\begin{bmatrix}
         2& 2\mathbf{i} \\
          -2\mathbf{i} &2
      \end{bmatrix}\right)=-4\mathbf{k},\\
  a_{31}^{\tiny\textcircled{\#},r}=&
=\frac{1}{8}\left({\rm{rdet}} _{1}\begin{bmatrix}
         0& -1+\mathbf{j} \\
          0 &2
      \end{bmatrix}+{\rm{rdet}} _{1}\begin{bmatrix}
         0& 0 \\
          -2\mathbf{i} &2
      \end{bmatrix}\right)=0,\\
 a_{12}^{\tiny\textcircled{\#},r}=&
=\frac{1}{8}\left({\rm{rdet}} _{1}\begin{bmatrix}
         2& 0 \\
          0 &-\mathbf{i}+\mathbf{k}
      \end{bmatrix}+{\rm{rdet}} _{2}\begin{bmatrix}
         -\mathbf{i}+\mathbf{k} &0 \\
          0 &2
      \end{bmatrix}\right)=-4\mathbf{i}+4\mathbf{k},\\
  a_{22}^{\tiny\textcircled{\#},r}=&
=\frac{1}{8}\left({\rm{rdet}} _{1}\begin{bmatrix}
         2& 0 \\
          -2\mathbf{k}&0
      \end{bmatrix}+{\rm{rdet}} _{2}\begin{bmatrix}
         0 &-2\mathbf{j} \\
          0 &2
      \end{bmatrix}\right)=0,\\
  a_{32}^{\tiny\textcircled{\#},r}=&
=\frac{1}{8}\left({\rm{rdet}} _{1}\begin{bmatrix}
         2& 0 \\
          0 &-1+\mathbf{j}
      \end{bmatrix}+{\rm{rdet}} _{2}\begin{bmatrix}
         -1+\mathbf{j} &0 \\
          0 &2
      \end{bmatrix}\right)=-4+4\mathbf{j},\\
   a_{13}^{\tiny\textcircled{\#},r}=&
=\frac{1}{8}\left({\rm{rdet}} _{1}\begin{bmatrix}
         2& 2\mathbf{i} \\
          0 &0
      \end{bmatrix}+{\rm{rdet}} _{2}\begin{bmatrix}
         2 &0 \\
          -\mathbf{i}+\mathbf{k} &2
      \end{bmatrix}\right)=0,\\
    a_{23}^{\tiny\textcircled{\#},r}=&
=\frac{1}{8}\left({\rm{rdet}} _{1}\begin{bmatrix}
         2& 2\mathbf{i} \\
         -2\mathbf{k} &-2\mathbf{j}
      \end{bmatrix}+{\rm{rdet}} _{2}\begin{bmatrix}
         2 &0 \\
         0 &-2\mathbf{j}
      \end{bmatrix}\right)=-4\mathbf{j},\\
              a_{33}^{\tiny\textcircled{\#},r}=&
=\frac{1}{8}\left({\rm{rdet}} _{1}\begin{bmatrix}
         2& 2\mathbf{i} \\
         0 &0
      \end{bmatrix}+{\rm{rdet}} _{2}\begin{bmatrix}
         2 &0 \\
         -1+\mathbf{j} &0
      \end{bmatrix}\right)=0.
\end{align*}
So, $${\bf A}^{\tiny\textcircled{\#}}=\begin{bmatrix}
      0& -0.5\mathbf{i}+0.5\mathbf{k}&0 \\
        -0.5\mathbf{k} &0&-0.5\mathbf{j} \\
       0&-0.5+0.5\mathbf{j}&0\\
      \end{bmatrix}.$$
By analogy, due to (\ref{eq:detrep_lepcorep_sim}), we have
$${\bf A}_{\tiny\textcircled{\#}}=\begin{bmatrix}
      0& -0.5\mathbf{i}&0 \\
        0.5\mathbf{i}-0.5\mathbf{k} &0&0.5-0.5\mathbf{j} \\
       0&0.5\mathbf{j}&0\\
      \end{bmatrix}.$$
      We can verify the results, for example,  by the  representation $(ii)$ from Lemma \ref{lem:rep_cor2} because by Theorem \ref{th:det_rep_mp}
       \begin{align*}
{\bf A}^{\dag}=
      \begin{bmatrix}
      0& -0.5\mathbf{i}&0 \\
        -0.5\mathbf{k} &0&-0.5\mathbf{j} \\
       0&0.5\mathbf{j}&0\\
      \end{bmatrix},~ \left({\bf A}_{\tiny\textcircled{\#}}\right)^{\dag}=
      \begin{bmatrix}
      0& -0.5\mathbf{i}+0.5\mathbf{k}&0 \\
        \mathbf{i} &0&-\mathbf{j} \\
       0&0.5+0.5\mathbf{j}&0\\
      \end{bmatrix},\end{align*}  and ${\bf Q}_{A}{\bf A}=\left({\bf A}_{\tiny\textcircled{\#}}\right)^{\dag}$.

\section{Conclusion}\label{sec:con}
\noindent  Notions of the core inverse, the core EP inverse, the  DMP and MPD inverses, and the  CMP inverse have been extended to quaternion matrices  in this paper. Due to noncommutativity of quaternions, these generalized inverses in quaternion matrices  have some features in comparison to complex matrices. We have obtained their determinantal representations within the framework of the theory of column-row determinants previously introduced by the author. 
As the special case,  their determinantal representations in complex matrices have been obtained as well.
A numerical example to illustrate the main result has given.


\begin{thebibliography}{40}
\bibitem{as}H. Aslaksen, Quaternionic determinants, Math. Intellig. 18(3) (1996) 57-65.

\bibitem{baks} O.M.~Baksalary, G.~Trenkler, Core inverse of matrices, Linear Multilinear
Algebra 58 (2010) 681-697.
\bibitem{baks1}O.M.~Baksalary, G.~Trenkler, On a generalized core inverse, Appl. Math. Comput.
236 (2014) 450-457.

\bibitem{bap}
R.B. Bapat,  K.P.S. Bhaskara Rao, K. Manjunatha Prasad,  Generalized inverses over integral domains,
Linear Algebra Appl. 140 (1990) 181-196.

\bibitem{rao}  K.P.S. Bhaskara Rao, Generalized inverses of matrices over integral domains, Linear Algebra Appl. 49 (1983) 179-189.



\bibitem{chen}
J.L. Chen, H.H. Zhu, P. Patri\'{c}io, Y.L. Zhang, Characterizations and representations of core and dual core inverses, Canad. Math. Bull. 60 (2017)
269-282.
\bibitem{coh} N.~Cohen, S.~De~Leo. The quaternionic determinant, Electron. J. Linear Algebra 7 (2000)  100-111.
\bibitem{gao}
 Y.F. Gao, J.L. Chen, Pseudo core inverses in rings with involution, Comm.
Algebra 46 (2018) 38-50.
\bibitem{gut}
A. Guterman, A. Herrero, N. Thome, New matrix partial order based
on spectrally orthogonal matrix
decomposition, Linear  Multilinear Algebra 64(3) (2016) 362-374.



\bibitem{fer1}
D.E. Ferreyra,    F.E. Levis,  N. Thome, Maximal classes of matrices determining generalized inverses, Appl. Math. Comput.  333 (2018) 42-52.
\bibitem{fer2}D.E. Ferreyra,    F.E. Levis,  N. Thome, Revisiting the core EP inverse and its extension to rectangular
matrices,  Quaest. Math. 41 (2) (2018) 265-281.



   \bibitem{kyr}
I.~Kyrchei, Analogs of the adjoint matrix for generalized inverses and corresponding Cramer rules, Linear  Multilinear Algebra 56(4) (2008)  453-469.
   \bibitem{kyr1}I.~Kyrchei, Explicit formulas for determinantal representations of the Drazin inverse solutions of some matrix and differential matrix equations, Appl. Math. Comput. 219 (2013)  7632-7644.


    \bibitem{kyr_nov}I.~Kyrchei, Cramer's rule for generalized inverse solutions. In: I. Kyrchei (Ed.), Advances in Linear Algebra Research, pp. 79--132,  Nova Sci. Publ., New York,  2015.


 \bibitem{kyr2} I.~Kyrchei, Cramer's rule for quaternionic systems of linear equations, J. Math. Sci. 155(6) (2008) 839--858.


 \bibitem{kyr3}    I.~Kyrchei, The theory of the column and row determinants in a quaternion linear algebra.  In: Albert R. Baswell (Ed.), Advances in Mathematics Research 15, pp. 301-359,  Nova Sci. Publ., New York, 2012.

\bibitem{kyr4}    I.~Kyrchei,
 Determinantal representations of the Moore-Penrose inverse over the quaternion skew field and corresponding Cramer's rules, Linear  Multilinear Algebra 59(4) (2011) 413-431.

 \bibitem{kyr5}I.~Kyrchei,
Determinantal representations of the Drazin inverse over the quaternion skew field with applications to some matrix equations, Appl. Math. Comput. 238 (2014) 193-207.

 \bibitem{kyr6}I.~Kyrchei,
Determinantal representations of the W-weighted Drazin inverse over the quaternion skew field, Appl. Math. Comput. 264 (2015) 453-465.

 \bibitem{kyr7}   I.~Kyrchei,
Explicit determinantal representation formulas of W-weighted Drazin inverse solutions of some matrix equations over the quaternion skew field, Math. Probl. Eng.  ID 8673809 (2016) 13 p.




\bibitem{kyr8}I.~Kyrchei, Explicit determinantal representation formulas for the solution of the two-sided restricted quaternionic matrix equation, J. Appl. Math. Comput. 58(1-2)  (2018) 335--365.



\bibitem{kyr9}I.~Kyrchei,
 Determinantal representations of the Drazin and W-weighted Drazin inverses over the quaternion skew field with applications.  In: Sandra Griffin (Ed.), Quaternions: Theory and Applications, pp.201-275, Nova Sci. Publ., New York,  2017.




\bibitem{kyr10}I.~Kyrchei,
 Weighted singular value decomposition and determinantal representations of the quaternion weighted Moore-Penrose inverse, Appl. Math. Comput.  309 (2017)  1-16.

\bibitem{kyr11}  I.I.~Kyrchei,   Determinantal representations of the quaternion weighted Moore-Penrose inverse and its applications, In: Albert R.B. (Ed.), Advances in Mathematics Research 23, pp.35--96, Nova Sci. Publ., New York, 2017.


 \bibitem{kyr12}I.~Kyrchei,
 Determinantal representations of solutions to systems of quaternion matrix equations, Adv. Appl. Clifford Algebras 28(1) (2018) 23.


 \bibitem{kyr13}     I.I.~Kyrchei,  Cramer's rules for Sylvester quaternion matrix equation and its special cases, Adv. Appl. Clifford Algebras 28(4) (2018) 90.

      \bibitem{kyr14}I.~Kyrchei, Determinantal representations of solutions and Hermitian solutions to some system of two-sided quaternion matrix equations,  Journal of Mathematics ID 6294672 (2018) 12 p.

    \bibitem{kyr15}I.~Kyrchei, Determinantal representations of solutions and Hermitian solutions to some system of two-sided quaternion matrix equations,  Abstract and Applied Analysis ID 5926832 (2019) 14 p.


 \bibitem{liu}
X. Liu, N. Cai, High-order iterative methods for the DMP inverse, Journal of Mathematics 8175935 (2018) 6 p.

\bibitem{mal1}
 S. Malik,  N. Thome,  On a new generalized inverse for matrices of an arbitrary index, Appl. Math. Comput. 226 (2014) 575-580.

\bibitem{meh} M. Mehdipour,  A. Salemi,  On a new generalized inverse of matrices, Linear Multilinear Algebra 66 (5) (2018) 1046-1053.
\bibitem{miel}J. Mielniczuk, Note on the core matrix partial ordering, Discuss. Math. Probab. Stat. 31 (2011) 71-75.
\bibitem{mos}
D. Mosi\'{c}, C. Deng, H. Ma, On a weighted core inverse in a ring with
involution, Comm.
Algebra 46(6) (2018) 2332-2345.

\bibitem{pras}
K.M. Prasad,  K.S. Mohana,  Core EP inverse, Linear Multilinear Algebra 62(3) (2014) 792-802.

\bibitem{pras1}
K.M. Prasad,  M.D. Raj,  Bordering method to compute Core-EP inverse, Spec. Matrices 6 (2018) 193-200.

  \bibitem{zh_arx}M. Zhou, J. Chen, T. Li, D. Wang, Three limit representations of the core-EP inverse, http://arxiv.org/abs/1804.05206v1.

\bibitem{rak}
 D.S. Raki\'{c}, N. \v{C}. Din\v{c}i\'{c}, D.S. Djordjevi\'{c}, Group, Moore-Penrose, core and
dual core inverse in rings with involution, Linear Algebra Appl. 463 (2014)
115-133.


\bibitem{sta1}  P.S.~Stanimirovi\'{c}, General determinantal representation of
  pseudoinverses of matrices. Mat. Vesnik. 48 (1996) 1-9.
\bibitem{sta2}
P.S.~Stanimirovi\'{c}, D.S. Djordjevic, Full-rank and determinantal representation of the Drazin inverse, Linear Algebra Appl. 311 (2000) 131-151.

\bibitem{song1}
 G.J.~Song,   Determinantal representations of the generalized inverses $A_{T,S}^{(2)}$ over the quaternion skew field with applications,  J. Appl. Math. Comput. 39   (2012) 201-220.


\bibitem{song2}G.J.~Song, Bott-Duffin inverse over the quaternion skew field with applications, J. Appl. Math. Comput. 41   (2013) 377-392.

\bibitem{song3} G.J.~Song, Q.W.~Wang, Condensed Cramer rule for some restricted quaternion
linear equations, Appl. Math. Comp. 218 (2011) 3110-3121.

 \bibitem{song5}
 G.J.~Song, Characterization of the W-weighted Drazin inverse over the quaternion
skew field with applications, Electron. J. Linear Algebra 26 (2013) 1-14.

\bibitem{song6} G.J.~Song,  Q.W.~Wang,   S.W. Yu, Cramer's rule for a system of quaternion matrix equations with applications, Appl. Math. Comp. 336 (2018) 490-499.

\bibitem{xu} S.Z. Xu, J.L. Chen, X.X. Zhang, New characterizations for core inverses in rings with involution, Front. Math. China 12 (2017) 231-246.






\bibitem{wang}
H.X. Wang, Core-EP decomposition and its applications, Linear Algebra
Appl. 508 (2016) 289-300.
































\end{thebibliography}
\end{document}